\documentclass[onefignum,onetabnum]{siamart171218}

\usepackage{lipsum}
\usepackage{amsfonts}
\usepackage{graphicx}
\usepackage{epstopdf}
\usepackage{algorithmic}
\ifpdf
  \DeclareGraphicsExtensions{.eps,.pdf,.png,.jpg}
\else
  \DeclareGraphicsExtensions{.eps}
\fi

\usepackage{todonotes}
\definecolor{notefarbe}{rgb}{0.9, 0.9,  0.982}

\usepackage{inputenc}
\inputencoding{utf8}
\makeatother

\usepackage{pgf}
\usepackage{pgfplots}
\usetikzlibrary{matrix}
\usepgfplotslibrary{groupplots}
\pgfplotsset{compat=newest}

\usepackage{graphicx}
\usepackage{amssymb}

\usepackage{ragged2e}
\usepackage{tabularx}
\usepackage{booktabs}
\usepackage{caption}

\usepackage{enumitem}
\usepackage{xspace}
\usepackage{mathtools}

\DeclarePairedDelimiterXPP\seq[1]{}{(}{)}{}{#1}

\DeclareMathOperator{\K}{\mathcal{K}}

\DeclareMathOperator{\C}{\mathcal{C}}

\newcommand{\ep}{\varepsilon}
\newcommand{\dd}{\ensuremath{\, \mathrm{d}}}

\newcommand{\etal}{\emph{et al.}}

\newcommand{\R}{\mathbb{R}}
\newcommand{\cX}{{\cal X}}

\newcommand{\lb}{\left(}
\newcommand{\rb}{\right)}

\newcolumntype{b}{X}
\newcolumntype{s}{>{\hsize=.5\hsize}X}
\newcolumntype{Y}{>{\RaggedRight\arraybackslash}X}
\newsiamremark{remark}{Remark}
\newsiamremark{hypothesis}{Hypothesis}
\crefname{hypothesis}{Hypothesis}{Hypotheses}
\newsiamthm{claim}{Claim}
\newsiamremark{example}{Example}
\newsiamremark{exampleSummary}{Example Summary}

\headers{Nonlocal Adhesion Models on Bounded Domains}{T. Hillen, A. Buttensch\"{o}n}

\title{Nonlocal Adhesion Models for Microorganisms on Bounded Domains
    \thanks{Submitted to the editors DATE.}}

\author{Thomas Hillen\thanks{Department of Mathematical and Statistical
Sciences, Centre for Mathematical Biology, University of Alberta, Edmonton
T6G~2G1, AB, Canada
  (\email{thillen@ualbert.ca}, \url{http://ualberta.ca/\string~thillen/}).}
\and Andreas Buttensch\"{o}n \thanks{Department of Mathematics, University of
British Columbia, Vancouver V6T~1Z2, BC, Canada (\email{abuttens@math.ubc.ca}, \url{www.buttenschoen.ca}).}}

\usepackage{amsopn}

\ifpdf
\hypersetup{%
  pdftitle={Nonlocal Adhesion Models for Microorganisms on Bounded Domains},
  pdfauthor={T. Hillen, A. Buttensch\"{o}n}
}
\fi

\begin{document}

\maketitle

\begin{abstract}
    In 2006 Armstrong, Painter and Sherratt formulated a non-local differential
    equation model for cell-cell adhesion. For the one dimensional case we
    derive various types of adhesive, repulsive, and no-flux boundary
    conditions. We prove local and global existence and uniqueness for the
    resulting integro-differential equations. In numerical simulations we
    consider adhesive, repulsive and neutral boundary conditions and we show
    that the solutions mimic known behavior of fluid adhesion to boundaries. In
    addition, we observe interior pattern formation due to cell-cell adhesion.
\end{abstract}



\section{Introduction}\label{S:Intro}

Adhesion between cells and other tissue components are fundamental in tissue
development (embryogenesis), and homeostasis and repair of tissues. Cellular
adhesion allow cells to self-organize by exerting forces on each other. A
single adhesive cell population, for instance, will aggregate to form sheets or aggregates,
while two cell populations ``sort'' into one of four cell-sorting
configurations, first
described by Steinberg \cite{Steinberg62a}. The regulation of cellular adhesion is critical in
both development and pathological conditions such as cancer. For
many cancers the loss of cell-cell cohesion is a pre-requiste for
cell invasion and subsequent metastasis formation. Due to their biological
importance, it is highly desirable to have accurate models of cellular adhesion
as part of standard modelling frameworks. Here we consider models of
the reaction-diffusion-taxis form, which are popular in the modelling of
biological tissues.

In 2006, Armstrong \etal\ \cite{Armstrong2006} proposed the first successful
continuum model of cellular adhesion. The novelty of this model is the use
of a non-local integral term to describe the adhesive forces between cells.
To introduce the model, we let $u(x, t)$ denote a cell density at spatial
location $x$ and time $t$, then on the real line the model is given by
the following non-local partial differential equations.
\begin{equation}\label{Eqn:ArmstrongModelIntro}
    u_t(x, t) = D u_{xx}(x, t) - \alpha \lb u(x, t)
                \int_{-R}^{R} H(u(x + r, t)) \Omega(r) \dd r \rb_{x},
\end{equation}
where $D$ is the diffusion coefficient, $\alpha$ the strength of homotypic cell
adhesion, $H(u)$ is a possibly nonlinear function describing the nature of the
adhesion force, $\Omega(r)$ is an odd function giving the
adhesion force's direction and $R$ the sensing radius of the cell.
The model~\eqref{Eqn:ArmstrongModelIntro} was derived from an underlying
stochastic random walk in \cite{Buttenschon2017}.

The novelty of model~\eqref{Eqn:ArmstrongModelIntro} is the integral term
modelling cell-cell adhesion. Intuitively, the integral term can be interpreted
as a tug-of-war, or a force balance causing cells to move in the direction
of largest adhesion force. Since other cells are required for the creation
of adhesive forces, it is easy to see that the non-local term causes cells
to aggregate. Furthermore the two cell population version of
model~\eqref{Eqn:ArmstrongModelIntro} is the first continuum model to replicated
the different cell-sorting experiments from Steinberg's
classical experiments~\cite{Armstrong2006}.

In biological systems, cellular adhesion features prominently in organism
development, wound-healing, and cancer invasion (metastasis). Therefore, it is
unsurprising that model~\eqref{Eqn:ArmstrongModelIntro} has been extensively
used to model cancer cell invasion
\cite{Gerisch2008,Sherratt2008,Gerisch2010,Painter2010,Chaplain2011,Andasari2012a,Domschke2014,Bitsouni2017a},
and developmental processes \cite{Armstrong2009,Painter2015}.
To allow for
numerical exploration, Gerisch et al.  \cite{Gerisch2010a} developed an efficient numerical
method for the integral term in~\eqref{Eqn:ArmstrongModelIntro}.
Finally, with the availability of controlled biological experiments
\cite{Murakawa2015,carrillo2019population} extended the adhesion
model~\eqref{Armstrong1} with density-dependent diffusion, and volume filling
to improve the model's fit to experimental data.

Existence results for the solutions of the non-local
equation~\eqref{Eqn:ArmstrongModelIntro} were developed in
\cite{Sherratt2008,Andasari2012a,Hillen2017}. Most significant is the general
work by Hillen~\etal \cite{Hillen2017}, who showed local and global existence of classical
solutions in unbounded domains. Finally, for small values of adhesion strength $\alpha$,
travelling wave solutions of the non-local adhesion model have been
described in \cite{Chunhua2013}.

All of the above mentioned models and results considered models on unbounded
or periodic domains, since this avoids defining the non-local adhesion
operator near boundaries. In this paper, we extend
model~\eqref{Eqn:ArmstrongModelIntro} to a bounded domain.
Our work is motivated by observations that adhesive or repulsive
cell-boundary interactions are significant during development. For instance,
repulsive membranes are required for correct organ placement in
zebrafish~\cite{Paksa2016}. In this work, we formulate different biological
boundary conditions for model~\eqref{Eqn:ArmstrongModelIntro}, describing adhesive, repulsive, or neutral boundary interactions. Due to the
non-locality we find that it is not sufficient to describe the non-local
operators behaviour on just the boundary, but its behaviour must be provided
in a boundary region.

Another class of non-local models for species aggregations are the so called {\it aggregation equations} \cite{fetecau2017swarm,wu2015nonlocal}. Here the non-local term arises through an interaction potential between different individuals. This interaction potential can describe long range attraction, short range repulsion and intermediate range alignment of species. There is an extensive mathematical theory related to the aggregation equations, and most of the results rely on the fact that the aggregation equations arise as gradient flow of a potential. Our adhesion model (\ref{Eqn:ArmstrongModelIntro})  does not have such a variational structure. The aggregation equations on a bounded domain have recently been studied in  \cite{fetecau2017swarm,wu2015nonlocal}. The boundary conditions are very similar to our adhesive and repulsive boundary conditions.

\subsection{Outline}
Starting from model~\eqref{Eqn:ArmstrongModelIntro} in divergence form,
in \cref{S:1} we formulate several biologically relevant boundary conditions. In particular,
we consider two cases (1) the adhesive flux is independent from the diffusive
flux and (2) the diffusive and adhesive flux depend on each other. In the case,
of independent fluxes, using semi-group theory, we develop a
local existence theory (\cref{S:existence}) and global existence
(\cref{sec:global_existence}) for the non-local adhesion model with
no-flux boundary conditions. In \cref{sec:numerics} we compare
numerical solutions of the adhesion model with different no-flux boundary
conditions to the periodic situation. We observe
boundary adhesion effects, similar to those known from thin film wetting of
glass boundaries. The case where the adhesive and diffusive fluxes are coupled
leads to non-trivial Robin-type boundary conditions. An existence theory for
those cases is much more involved and left for future research. In
\cref{sec:conclusion} we provide
some concluding remarks, and outlooks for future work.

\section{Boundary conditions for non-local operators}\label{S:1}

We consider the one-dimensional Armstrong adhesion model on the interval
$[0,L]$ with sensing radius $0<R<\frac{L}{2}$.
\begin{equation}\label{Armstrong1}
    u_t(x,t) = D u_{xx}(x,t) -\alpha\lb u(x,t) \int_{E(x)} H(u(x+r,t))\Omega(r) \dd r \rb_x ,
\end{equation}
and we define the non-local integral operator as
\begin{equation}\label{eq:K}
    \K[u](x,t) = \int_{E(x)} H(u(x+r,t) \Omega(r) \dd r.
\end{equation}
The domain of integration $E(x)\subset [-R,R]$ is chosen to ensure that the
integrand does not reach outside of the domain $[0,L]$, and it is called the
{\it sampling domain}. The sampling domain is not unique and we give several
examples in \cref{subsec:SensingExamples}.

To address the boundary conditions we consider the particle flux
\begin{equation}\label{flux}
    J(x,t) = D u_x(x,t) - \alpha u(x,t) \K[u](x,t).
\end{equation}
Our first goal is to formulate no-flux boundary conditions i.e. $J(x,t) = 0$
for $x = 0, L$. We consider two different cases; (1) the diffusive flux and the
adhesive flux are independent on the boundary (2) the diffusive and adhesive
flux depend on each other.

\paragraph{Independent case}
If we stipulate that the diffusive and adhesive component of the flux are
independently zero on the boundary, then the following are a suitable set of
boundary conditions for~\eqref{Armstrong1}.
\begin{eqnarray}
    && u_x(0,t) = u_x(L,t) =0,\label{nofluxdiffusion}\\
    && \K[u](0) = \K[u](L) = 0. \label{nofluxK}
\end{eqnarray}
The first condition \eqref{nofluxdiffusion} is a condition on the solution
$u(x,t)$ and we can include this into the right choice of function space. The
second condition \eqref{nofluxK}, however, should be seen as a condition on the
non-local operator $\K$ i.e.\ condition~\eqref{nofluxK} must hold for any $u$.
In other words, in this no-flux situation we only
consider non-local operators $\K$ that satisfy \eqref{nofluxK}. We give explicit
examples later.

\paragraph{Dependent case}
If we want to describe adhesion to or repulsion from the
boundaries, we can relax the above conditions on $\K$. For example if we assume
\begin{equation}\label{adhesivebc}
    \K[u](0)<0, \qquad \K[u](L)>0,
\end{equation}
then we have net flow towards the boundaries, i.e.\ an adhesive boundary, while
\begin{equation}\label{repulsivebc}
    \K[u](0)>0, \qquad \K[u](L)<0,
\end{equation}
denote repulsive boundary conditions. However, to ensure that the total particle
flux $J(x,t)$ is zero on the boundary we require that the diffusive flux component
counter act the adhesive component, that is
\begin{equation}\label{Robin}
    D u_x(x, t) = \alpha u(x,t) \K[u](x,t),\qquad\mbox{for } x = 0, L.
\end{equation}
This results in a non-local boundary condition of Robin type.

\subsection{Examples}\label{subsec:SensingExamples}
We consider several examples of sensing domains $E(x)$ for use in the non-local operator
$\K[u]$ defined in equation~\eqref{eq:K}. The examples are summarized in Table~\ref{tab:examples}.

\begin{table}
\small
\setlength{\extrarowheight}{5pt}
\begin{tabular}{|l|c|c|c|}
\hline
 \bf Case & $\K[u]$ & $f_1(x)$ & $ f_2(x)$ \\
 \hline
 \hline
periodic & $\int_{E(x)} H(u) \Omega(r) \dd r $ & $f_1 = -R $ & $f_2 =R$\\
\hline
naive & $ \int_{E(x)} H(u) \Omega(r) \dd r $ &
$f_1 = \begin{cases} -x, & I_1\\ -R, &  I_2\end{cases} $ &
$f_2 = \begin{cases} R, &   I_3\\ L-x, &I_4\end{cases}$ \\ \hline
non-flux & $ \int_{E(x)} H(u) \Omega(r) \dd r $ &
    $f_1 = \begin{cases} R-2x, & I_1 \\ -R, & I_2 \end{cases} $ &
    $f_2 = \begin{cases} R, & I_3\\ 2L-R-2x, & I_4\end{cases}$\\
\hline
weakly & $\int_{E(x)} H(u) \Omega(r) \dd r + a_0 + a_L$ & $f_1=$naive & $f_2 =$ naive\\
adhesive & $a_0 = \beta^0 \int_{-R}^{-x} \Omega(r) \dd r$ & & \\
& $ a_L = \beta^L\int_{L-x}^R \Omega(r) \dd r$ & & \\
\hline
\end{tabular}

\caption{The different cases of suitable boundary conditions on $[0,L]$. The sensing slice is defined as $E(x) = \{ r\in[-R,R]: f_1(x) \leq r \leq f_2(x) \}$. The abbreviations $I_1, I_2, I_3, I_4$ stand for
$x\in[ 0,R], x\in(R,L], x\in[0,L-R], x\in (L-R,L]$, respectively.}\label{tab:examples}
\end{table}

\begin{example}[periodic]\label{Example:Periodic}
    The periodic case is special, since $x=0$ and $x=L$ are identified. Any integral
    over a domain of length $2R$ is well defined. This case is
    included in our framework with the choice of sampling domain of
\[
    E_1(x) = [-R, R].
\]
and periodic boundary conditions.
\end{example}

\begin{example}[naive]\label{Example:Naive}
    The first idea of a well defined integral
    operator~\eqref{eq:K} on $[0,L]$ is to remove those parts of the integration
    that leave the domain. This can be achieved through the sampling slice
    \[
        E_2(x) = \{ r\in[-R,R], f_1(x)\leq r \leq  f_2(x) \}
    \]
    with
    \[
        f_1(x) = \left\{\begin{array}{ll} -x, &  x\in[0,R]\\ -R, &  x\in(R,L] \end{array}\right.,
        \quad
        f_2(x) = \left\{\begin{array}{ll} R, &  x\in[0,L-R) \\ L-x, & x\in[L-R,L]\end{array}\right..
    \]
    At the left boundary we obtain
    \[
        \K[u](0) = \int_0^R H(u(x+r,t)) \Omega(r) \dd r \geq 0,
    \]
    which is non-negative for positive $H$ and $\Omega$. Similarly we find
    $\K[u](L)\leq 0$. In this situation cells at the boundary are attracted by cells in the interior
    with no interaction with the wall. Hence a net flow away from the boundary is created. By equation~\eqref{repulsivebc}, we classify these naive boundary conditions
    as repulsive. Further, this implies that $E_2(x)$ are not suitable to ensure that
    $\K[u]$ are zero on the boundary for all $u$. For a pictorial representation of the sampling
    domain $E_2(x)$ see \cref{Fig:DomainOfIntegration}.
\end{example}

\begin{example}[no-flux]\label{Example:NoFlux}
    In this example we choose the sampling domain such that the sampling domain $E(x)$ is a
    set of measure zero for $x = 0, L$, thus ensuring that $\K[u]$ is zero on the boundary.
    We let
    \[
        E_3(x) = \{ r\in[-R,R], f_1(x)\leq r \leq  f_2(x) \},
    \]
    where now
    \[
        f_1(x) = \left\{\begin{array}{ll} R-2x, &  x\in[0,R]\\ -R, &  x\in(R,L] \end{array}\right.,
        \quad
        f_2(x) = \left\{\begin{array}{ll} R, &  x\in[0,L-R) \\ 2L-R-2x, &
        x\in[L-R,L]\end{array}\right..
    \]
    In this case we obtain on the left boundary
    \[
        \K[u](0) = \int_R^R H(u(x+r,t)) \Omega(r) \dd r = 0,
    \]
    and $\K[u](L) =0$, and hence condition~\eqref{nofluxK} is satisfied. This makes
    $E_3(x)$ a suitable sampling domain for the independent no-flux boundary conditions.
    In this situation cell protrusions which hit the boundary fold back onto the cell itself,
    thus neutralizing the cell's adhesion molecules (see \cref{fig:bioboundary1}~(A)).
\end{example}

\begin{figure}[ht]\centering
    \includegraphics[width=10cm]{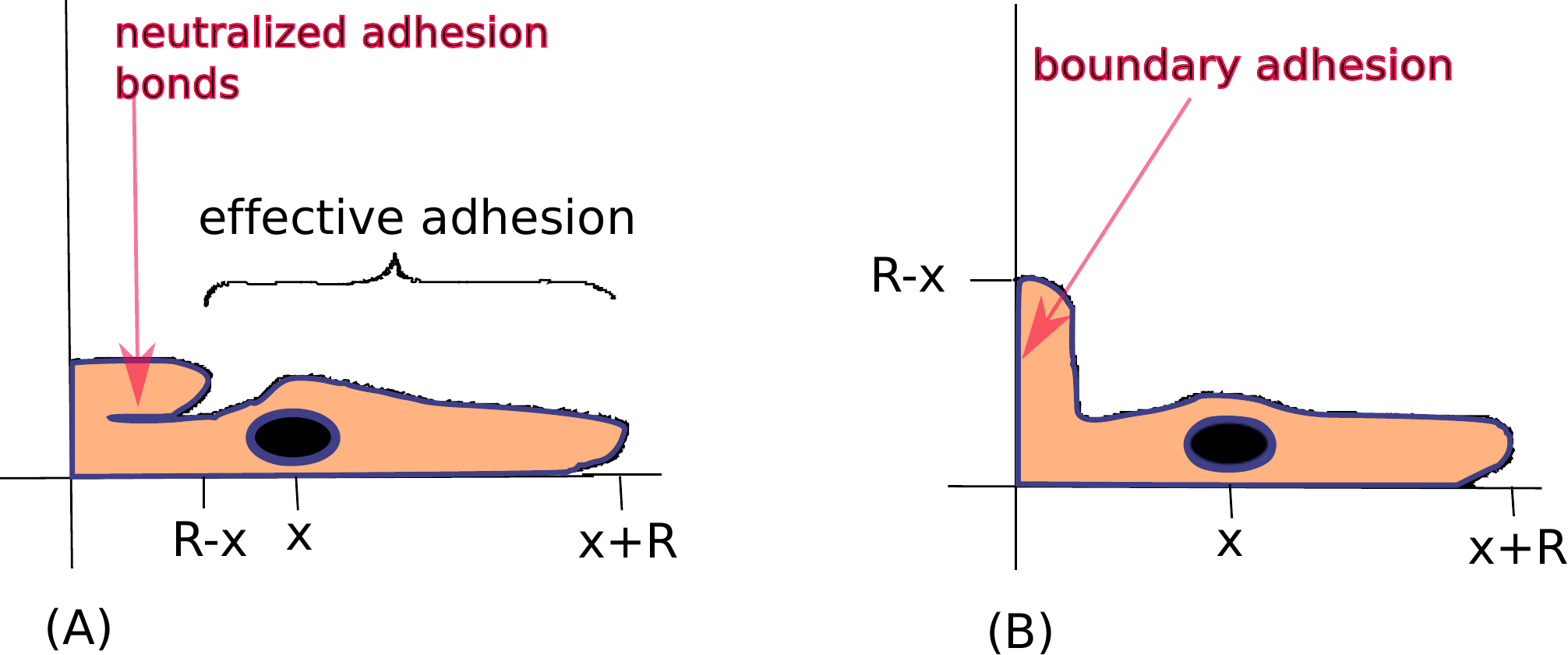}
    \caption{(A): The filopodia of cell are reflected or stopped at the boundary.
    As a result the cell starts to form adhesion bonds with itself, which are then
    not contributing to the net adhesion force. Note that only one cell is shown in
    this sketch.  (B) The weak adhesive case. Cells make contact to the boundary in
    a well balanced way, such that the net flux is still zero.\label{fig:bioboundary1}}
\end{figure}

Inspired by the previous examples we define a whole class of suitable sampling domain $E(x)$ as follows.
\begin{definition}[Sampling domain]\label{def:slice}
\begin{enumerate}
\item Two continuous functions $f_{1,2}:D\to \R$ define a suitable sampling domain $E(x)$ if they satisfy
\begin{enumerate}
\item $-R\leq f_1(x) \leq f_2(x)\leq R $ for all $x\in[0,L]$.
\item $f_1(x) = -R$ for $x\in[R, L]$
\item $f_2(x) = R$ for $x\in[0,L-R]$.
\item $f_1(x)$ and $f_2(x)$ are non-increasing and have uniformly bounded one-sided derivatives.
\end{enumerate}
\item A suitable sampling domain $E(x)$ satisfies condition~\eqref{nofluxK} if in addition
\begin{enumerate}[resume]
\item $f_1(0) = R \quad \mbox{and} \quad  f_2(L) = -R$. 
\end{enumerate}
\end{enumerate}
\end{definition}
It is straight forward to check that all our sampling domain are suitable. However, only
$E_3(x)$ satisfies condition~\eqref{nofluxK}. In Figures~\ref{Fig:DomainOfIntegration} and~\ref{Fig:NoFluxDomainOfIntegration} we show two examples of sampling domains over the whole
domain $[0, L]$.

\begin{minipage}[t]{0.45\textwidth}\centering
    \resizebox{0.95\textwidth}{!}{%
\begin{tikzpicture}

\begin{axis}[
title={The Sampling Domain},
xlabel={Space},
ylabel={Heading},
xmin=-0.1, xmax=3.1,
ymin=-1.3, ymax=1.3,
width=12.6cm,
height=8cm,
xtick={0,1.5,3},
xticklabels={$0$,$L/2$,$L$},
ytick={-1,0,1},
yticklabels={$-R$,$0$,$R$},
tick align=outside,
tick pos=left,
x grid style={lightgray!92.026143790849673!black},
y grid style={lightgray!92.026143790849673!black}
]
\path [draw=black, draw opacity=0.75, line width=0.83630769230769231pt, dash pattern=on 2pt off 3pt] (axis cs:0,-1.3)
--(axis cs:0,1.3);

\path [draw=black, draw opacity=0.75, line width=0.83630769230769231pt, dash pattern=on 2pt off 3pt] (axis cs:3,-1.3)
--(axis cs:3,1.3);

\path [draw=black, draw opacity=0.75, line width=0.83630769230769231pt, dash pattern=on 2pt off 3pt] (axis cs:-0.5,-1)
--(axis cs:3.5,-1);

\path [draw=black, draw opacity=0.75, line width=0.83630769230769231pt, dash pattern=on 2pt off 3pt] (axis cs:-0.5,0)
--(axis cs:3.5,0);

\path [draw=black, draw opacity=0.75, line width=0.83630769230769231pt, dash pattern=on 2pt off 3pt] (axis cs:-0.5,1)
--(axis cs:3.5,1);

\path [draw=lightgray, fill=lightgray] (axis cs:0,1)
--(axis cs:0,0)
--(axis cs:0.015625,-0.015625)
--(axis cs:0.03125,-0.03125)
--(axis cs:0.046875,-0.046875)
--(axis cs:0.0625,-0.0625)
--(axis cs:0.078125,-0.078125)
--(axis cs:0.09375,-0.09375)
--(axis cs:0.109375,-0.109375)
--(axis cs:0.125,-0.125)
--(axis cs:0.140625,-0.140625)
--(axis cs:0.15625,-0.15625)
--(axis cs:0.171875,-0.171875)
--(axis cs:0.1875,-0.1875)
--(axis cs:0.203125,-0.203125)
--(axis cs:0.21875,-0.21875)
--(axis cs:0.234375,-0.234375)
--(axis cs:0.25,-0.25)
--(axis cs:0.265625,-0.265625)
--(axis cs:0.28125,-0.28125)
--(axis cs:0.296875,-0.296875)
--(axis cs:0.3125,-0.3125)
--(axis cs:0.328125,-0.328125)
--(axis cs:0.34375,-0.34375)
--(axis cs:0.359375,-0.359375)
--(axis cs:0.375,-0.375)
--(axis cs:0.390625,-0.390625)
--(axis cs:0.40625,-0.40625)
--(axis cs:0.421875,-0.421875)
--(axis cs:0.4375,-0.4375)
--(axis cs:0.453125,-0.453125)
--(axis cs:0.46875,-0.46875)
--(axis cs:0.484375,-0.484375)
--(axis cs:0.5,-0.5)
--(axis cs:0.515625,-0.515625)
--(axis cs:0.53125,-0.53125)
--(axis cs:0.546875,-0.546875)
--(axis cs:0.5625,-0.5625)
--(axis cs:0.578125,-0.578125)
--(axis cs:0.59375,-0.59375)
--(axis cs:0.609375,-0.609375)
--(axis cs:0.625,-0.625)
--(axis cs:0.640625,-0.640625)
--(axis cs:0.65625,-0.65625)
--(axis cs:0.671875,-0.671875)
--(axis cs:0.6875,-0.6875)
--(axis cs:0.703125,-0.703125)
--(axis cs:0.71875,-0.71875)
--(axis cs:0.734375,-0.734375)
--(axis cs:0.75,-0.75)
--(axis cs:0.765625,-0.765625)
--(axis cs:0.78125,-0.78125)
--(axis cs:0.796875,-0.796875)
--(axis cs:0.8125,-0.8125)
--(axis cs:0.828125,-0.828125)
--(axis cs:0.84375,-0.84375)
--(axis cs:0.859375,-0.859375)
--(axis cs:0.875,-0.875)
--(axis cs:0.890625,-0.890625)
--(axis cs:0.90625,-0.90625)
--(axis cs:0.921875,-0.921875)
--(axis cs:0.9375,-0.9375)
--(axis cs:0.953125,-0.953125)
--(axis cs:0.96875,-0.96875)
--(axis cs:0.984375,-0.984375)
--(axis cs:1,-1)
--(axis cs:1.015625,-1)
--(axis cs:1.03125,-1)
--(axis cs:1.046875,-1)
--(axis cs:1.0625,-1)
--(axis cs:1.078125,-1)
--(axis cs:1.09375,-1)
--(axis cs:1.109375,-1)
--(axis cs:1.125,-1)
--(axis cs:1.140625,-1)
--(axis cs:1.15625,-1)
--(axis cs:1.171875,-1)
--(axis cs:1.1875,-1)
--(axis cs:1.203125,-1)
--(axis cs:1.21875,-1)
--(axis cs:1.234375,-1)
--(axis cs:1.25,-1)
--(axis cs:1.265625,-1)
--(axis cs:1.28125,-1)
--(axis cs:1.296875,-1)
--(axis cs:1.3125,-1)
--(axis cs:1.328125,-1)
--(axis cs:1.34375,-1)
--(axis cs:1.359375,-1)
--(axis cs:1.375,-1)
--(axis cs:1.390625,-1)
--(axis cs:1.40625,-1)
--(axis cs:1.421875,-1)
--(axis cs:1.4375,-1)
--(axis cs:1.453125,-1)
--(axis cs:1.46875,-1)
--(axis cs:1.484375,-1)
--(axis cs:1.5,-1)
--(axis cs:1.515625,-1)
--(axis cs:1.53125,-1)
--(axis cs:1.546875,-1)
--(axis cs:1.5625,-1)
--(axis cs:1.578125,-1)
--(axis cs:1.59375,-1)
--(axis cs:1.609375,-1)
--(axis cs:1.625,-1)
--(axis cs:1.640625,-1)
--(axis cs:1.65625,-1)
--(axis cs:1.671875,-1)
--(axis cs:1.6875,-1)
--(axis cs:1.703125,-1)
--(axis cs:1.71875,-1)
--(axis cs:1.734375,-1)
--(axis cs:1.75,-1)
--(axis cs:1.765625,-1)
--(axis cs:1.78125,-1)
--(axis cs:1.796875,-1)
--(axis cs:1.8125,-1)
--(axis cs:1.828125,-1)
--(axis cs:1.84375,-1)
--(axis cs:1.859375,-1)
--(axis cs:1.875,-1)
--(axis cs:1.890625,-1)
--(axis cs:1.90625,-1)
--(axis cs:1.921875,-1)
--(axis cs:1.9375,-1)
--(axis cs:1.953125,-1)
--(axis cs:1.96875,-1)
--(axis cs:1.984375,-1)
--(axis cs:2,-1)
--(axis cs:2.015625,-1)
--(axis cs:2.03125,-1)
--(axis cs:2.046875,-1)
--(axis cs:2.0625,-1)
--(axis cs:2.078125,-1)
--(axis cs:2.09375,-1)
--(axis cs:2.109375,-1)
--(axis cs:2.125,-1)
--(axis cs:2.140625,-1)
--(axis cs:2.15625,-1)
--(axis cs:2.171875,-1)
--(axis cs:2.1875,-1)
--(axis cs:2.203125,-1)
--(axis cs:2.21875,-1)
--(axis cs:2.234375,-1)
--(axis cs:2.25,-1)
--(axis cs:2.265625,-1)
--(axis cs:2.28125,-1)
--(axis cs:2.296875,-1)
--(axis cs:2.3125,-1)
--(axis cs:2.328125,-1)
--(axis cs:2.34375,-1)
--(axis cs:2.359375,-1)
--(axis cs:2.375,-1)
--(axis cs:2.390625,-1)
--(axis cs:2.40625,-1)
--(axis cs:2.421875,-1)
--(axis cs:2.4375,-1)
--(axis cs:2.453125,-1)
--(axis cs:2.46875,-1)
--(axis cs:2.484375,-1)
--(axis cs:2.5,-1)
--(axis cs:2.515625,-1)
--(axis cs:2.53125,-1)
--(axis cs:2.546875,-1)
--(axis cs:2.5625,-1)
--(axis cs:2.578125,-1)
--(axis cs:2.59375,-1)
--(axis cs:2.609375,-1)
--(axis cs:2.625,-1)
--(axis cs:2.640625,-1)
--(axis cs:2.65625,-1)
--(axis cs:2.671875,-1)
--(axis cs:2.6875,-1)
--(axis cs:2.703125,-1)
--(axis cs:2.71875,-1)
--(axis cs:2.734375,-1)
--(axis cs:2.75,-1)
--(axis cs:2.765625,-1)
--(axis cs:2.78125,-1)
--(axis cs:2.796875,-1)
--(axis cs:2.8125,-1)
--(axis cs:2.828125,-1)
--(axis cs:2.84375,-1)
--(axis cs:2.859375,-1)
--(axis cs:2.875,-1)
--(axis cs:2.890625,-1)
--(axis cs:2.90625,-1)
--(axis cs:2.921875,-1)
--(axis cs:2.9375,-1)
--(axis cs:2.953125,-1)
--(axis cs:2.96875,-1)
--(axis cs:2.984375,-1)
--(axis cs:3,-1)
--(axis cs:3,0)
--(axis cs:3,0)
--(axis cs:2.984375,0.015625)
--(axis cs:2.96875,0.03125)
--(axis cs:2.953125,0.046875)
--(axis cs:2.9375,0.0625)
--(axis cs:2.921875,0.078125)
--(axis cs:2.90625,0.09375)
--(axis cs:2.890625,0.109375)
--(axis cs:2.875,0.125)
--(axis cs:2.859375,0.140625)
--(axis cs:2.84375,0.15625)
--(axis cs:2.828125,0.171875)
--(axis cs:2.8125,0.1875)
--(axis cs:2.796875,0.203125)
--(axis cs:2.78125,0.21875)
--(axis cs:2.765625,0.234375)
--(axis cs:2.75,0.25)
--(axis cs:2.734375,0.265625)
--(axis cs:2.71875,0.28125)
--(axis cs:2.703125,0.296875)
--(axis cs:2.6875,0.3125)
--(axis cs:2.671875,0.328125)
--(axis cs:2.65625,0.34375)
--(axis cs:2.640625,0.359375)
--(axis cs:2.625,0.375)
--(axis cs:2.609375,0.390625)
--(axis cs:2.59375,0.40625)
--(axis cs:2.578125,0.421875)
--(axis cs:2.5625,0.4375)
--(axis cs:2.546875,0.453125)
--(axis cs:2.53125,0.46875)
--(axis cs:2.515625,0.484375)
--(axis cs:2.5,0.5)
--(axis cs:2.484375,0.515625)
--(axis cs:2.46875,0.53125)
--(axis cs:2.453125,0.546875)
--(axis cs:2.4375,0.5625)
--(axis cs:2.421875,0.578125)
--(axis cs:2.40625,0.59375)
--(axis cs:2.390625,0.609375)
--(axis cs:2.375,0.625)
--(axis cs:2.359375,0.640625)
--(axis cs:2.34375,0.65625)
--(axis cs:2.328125,0.671875)
--(axis cs:2.3125,0.6875)
--(axis cs:2.296875,0.703125)
--(axis cs:2.28125,0.71875)
--(axis cs:2.265625,0.734375)
--(axis cs:2.25,0.75)
--(axis cs:2.234375,0.765625)
--(axis cs:2.21875,0.78125)
--(axis cs:2.203125,0.796875)
--(axis cs:2.1875,0.8125)
--(axis cs:2.171875,0.828125)
--(axis cs:2.15625,0.84375)
--(axis cs:2.140625,0.859375)
--(axis cs:2.125,0.875)
--(axis cs:2.109375,0.890625)
--(axis cs:2.09375,0.90625)
--(axis cs:2.078125,0.921875)
--(axis cs:2.0625,0.9375)
--(axis cs:2.046875,0.953125)
--(axis cs:2.03125,0.96875)
--(axis cs:2.015625,0.984375)
--(axis cs:2,1)
--(axis cs:1.984375,1)
--(axis cs:1.96875,1)
--(axis cs:1.953125,1)
--(axis cs:1.9375,1)
--(axis cs:1.921875,1)
--(axis cs:1.90625,1)
--(axis cs:1.890625,1)
--(axis cs:1.875,1)
--(axis cs:1.859375,1)
--(axis cs:1.84375,1)
--(axis cs:1.828125,1)
--(axis cs:1.8125,1)
--(axis cs:1.796875,1)
--(axis cs:1.78125,1)
--(axis cs:1.765625,1)
--(axis cs:1.75,1)
--(axis cs:1.734375,1)
--(axis cs:1.71875,1)
--(axis cs:1.703125,1)
--(axis cs:1.6875,1)
--(axis cs:1.671875,1)
--(axis cs:1.65625,1)
--(axis cs:1.640625,1)
--(axis cs:1.625,1)
--(axis cs:1.609375,1)
--(axis cs:1.59375,1)
--(axis cs:1.578125,1)
--(axis cs:1.5625,1)
--(axis cs:1.546875,1)
--(axis cs:1.53125,1)
--(axis cs:1.515625,1)
--(axis cs:1.5,1)
--(axis cs:1.484375,1)
--(axis cs:1.46875,1)
--(axis cs:1.453125,1)
--(axis cs:1.4375,1)
--(axis cs:1.421875,1)
--(axis cs:1.40625,1)
--(axis cs:1.390625,1)
--(axis cs:1.375,1)
--(axis cs:1.359375,1)
--(axis cs:1.34375,1)
--(axis cs:1.328125,1)
--(axis cs:1.3125,1)
--(axis cs:1.296875,1)
--(axis cs:1.28125,1)
--(axis cs:1.265625,1)
--(axis cs:1.25,1)
--(axis cs:1.234375,1)
--(axis cs:1.21875,1)
--(axis cs:1.203125,1)
--(axis cs:1.1875,1)
--(axis cs:1.171875,1)
--(axis cs:1.15625,1)
--(axis cs:1.140625,1)
--(axis cs:1.125,1)
--(axis cs:1.109375,1)
--(axis cs:1.09375,1)
--(axis cs:1.078125,1)
--(axis cs:1.0625,1)
--(axis cs:1.046875,1)
--(axis cs:1.03125,1)
--(axis cs:1.015625,1)
--(axis cs:1,1)
--(axis cs:0.984375,1)
--(axis cs:0.96875,1)
--(axis cs:0.953125,1)
--(axis cs:0.9375,1)
--(axis cs:0.921875,1)
--(axis cs:0.90625,1)
--(axis cs:0.890625,1)
--(axis cs:0.875,1)
--(axis cs:0.859375,1)
--(axis cs:0.84375,1)
--(axis cs:0.828125,1)
--(axis cs:0.8125,1)
--(axis cs:0.796875,1)
--(axis cs:0.78125,1)
--(axis cs:0.765625,1)
--(axis cs:0.75,1)
--(axis cs:0.734375,1)
--(axis cs:0.71875,1)
--(axis cs:0.703125,1)
--(axis cs:0.6875,1)
--(axis cs:0.671875,1)
--(axis cs:0.65625,1)
--(axis cs:0.640625,1)
--(axis cs:0.625,1)
--(axis cs:0.609375,1)
--(axis cs:0.59375,1)
--(axis cs:0.578125,1)
--(axis cs:0.5625,1)
--(axis cs:0.546875,1)
--(axis cs:0.53125,1)
--(axis cs:0.515625,1)
--(axis cs:0.5,1)
--(axis cs:0.484375,1)
--(axis cs:0.46875,1)
--(axis cs:0.453125,1)
--(axis cs:0.4375,1)
--(axis cs:0.421875,1)
--(axis cs:0.40625,1)
--(axis cs:0.390625,1)
--(axis cs:0.375,1)
--(axis cs:0.359375,1)
--(axis cs:0.34375,1)
--(axis cs:0.328125,1)
--(axis cs:0.3125,1)
--(axis cs:0.296875,1)
--(axis cs:0.28125,1)
--(axis cs:0.265625,1)
--(axis cs:0.25,1)
--(axis cs:0.234375,1)
--(axis cs:0.21875,1)
--(axis cs:0.203125,1)
--(axis cs:0.1875,1)
--(axis cs:0.171875,1)
--(axis cs:0.15625,1)
--(axis cs:0.140625,1)
--(axis cs:0.125,1)
--(axis cs:0.109375,1)
--(axis cs:0.09375,1)
--(axis cs:0.078125,1)
--(axis cs:0.0625,1)
--(axis cs:0.046875,1)
--(axis cs:0.03125,1)
--(axis cs:0.015625,1)
--(axis cs:0,1)
--cycle;

\path [draw=darkgray, fill=darkgray] (axis cs:2.203125,0.796875)
--(axis cs:2.203125,-1)
--(axis cs:2.21875,-1)
--(axis cs:2.234375,-1)
--(axis cs:2.234375,0.765625)
--(axis cs:2.234375,0.765625)
--(axis cs:2.21875,0.78125)
--(axis cs:2.203125,0.796875)
--cycle;

\path [draw=black, line width=1.6726153846153846pt] (axis cs:0,0)
--(axis cs:0,1);

\path [draw=black, line width=1.6726153846153846pt] (axis cs:3,-1)
--(axis cs:3,0);

\path [draw=black, line width=0.83630769230769231pt] (axis cs:2.203125,-1)
--(axis cs:2.203125,0.796875);

\path [draw=black, line width=0.83630769230769231pt] (axis cs:2.234375,-1)
--(axis cs:2.234375,0.765625);

\addplot [line width=1.6726153846153846pt, black, forget plot]
table {%
0 1
0.015625 1
0.03125 1
0.046875 1
0.0625 1
0.078125 1
0.09375 1
0.109375 1
0.125 1
0.140625 1
0.15625 1
0.171875 1
0.1875 1
0.203125 1
0.21875 1
0.234375 1
0.25 1
0.265625 1
0.28125 1
0.296875 1
0.3125 1
0.328125 1
0.34375 1
0.359375 1
0.375 1
0.390625 1
0.40625 1
0.421875 1
0.4375 1
0.453125 1
0.46875 1
0.484375 1
0.5 1
0.515625 1
0.53125 1
0.546875 1
0.5625 1
0.578125 1
0.59375 1
0.609375 1
0.625 1
0.640625 1
0.65625 1
0.671875 1
0.6875 1
0.703125 1
0.71875 1
0.734375 1
0.75 1
0.765625 1
0.78125 1
0.796875 1
0.8125 1
0.828125 1
0.84375 1
0.859375 1
0.875 1
0.890625 1
0.90625 1
0.921875 1
0.9375 1
0.953125 1
0.96875 1
0.984375 1
1 1
1.015625 1
1.03125 1
1.046875 1
1.0625 1
1.078125 1
1.09375 1
1.109375 1
1.125 1
1.140625 1
1.15625 1
1.171875 1
1.1875 1
1.203125 1
1.21875 1
1.234375 1
1.25 1
1.265625 1
1.28125 1
1.296875 1
1.3125 1
1.328125 1
1.34375 1
1.359375 1
1.375 1
1.390625 1
1.40625 1
1.421875 1
1.4375 1
1.453125 1
1.46875 1
1.484375 1
1.5 1
1.515625 1
1.53125 1
1.546875 1
1.5625 1
1.578125 1
1.59375 1
1.609375 1
1.625 1
1.640625 1
1.65625 1
1.671875 1
1.6875 1
1.703125 1
1.71875 1
1.734375 1
1.75 1
1.765625 1
1.78125 1
1.796875 1
1.8125 1
1.828125 1
1.84375 1
1.859375 1
1.875 1
1.890625 1
1.90625 1
1.921875 1
1.9375 1
1.953125 1
1.96875 1
1.984375 1
2 1
2.015625 0.984375
2.03125 0.96875
2.046875 0.953125
2.0625 0.9375
2.078125 0.921875
2.09375 0.90625
2.109375 0.890625
2.125 0.875
2.140625 0.859375
2.15625 0.84375
2.171875 0.828125
2.1875 0.8125
2.203125 0.796875
2.21875 0.78125
2.234375 0.765625
2.25 0.75
2.265625 0.734375
2.28125 0.71875
2.296875 0.703125
2.3125 0.6875
2.328125 0.671875
2.34375 0.65625
2.359375 0.640625
2.375 0.625
2.390625 0.609375
2.40625 0.59375
2.421875 0.578125
2.4375 0.5625
2.453125 0.546875
2.46875 0.53125
2.484375 0.515625
2.5 0.5
2.515625 0.484375
2.53125 0.46875
2.546875 0.453125
2.5625 0.4375
2.578125 0.421875
2.59375 0.40625
2.609375 0.390625
2.625 0.375
2.640625 0.359375
2.65625 0.34375
2.671875 0.328125
2.6875 0.3125
2.703125 0.296875
2.71875 0.28125
2.734375 0.265625
2.75 0.25
2.765625 0.234375
2.78125 0.21875
2.796875 0.203125
2.8125 0.1875
2.828125 0.171875
2.84375 0.15625
2.859375 0.140625
2.875 0.125
2.890625 0.109375
2.90625 0.09375
2.921875 0.078125
2.9375 0.0625
2.953125 0.046875
2.96875 0.03125
2.984375 0.015625
3 0
};
\addplot [line width=1.6726153846153846pt, black, forget plot]
table {%
0 -0
0.015625 -0.015625
0.03125 -0.03125
0.046875 -0.046875
0.0625 -0.0625
0.078125 -0.078125
0.09375 -0.09375
0.109375 -0.109375
0.125 -0.125
0.140625 -0.140625
0.15625 -0.15625
0.171875 -0.171875
0.1875 -0.1875
0.203125 -0.203125
0.21875 -0.21875
0.234375 -0.234375
0.25 -0.25
0.265625 -0.265625
0.28125 -0.28125
0.296875 -0.296875
0.3125 -0.3125
0.328125 -0.328125
0.34375 -0.34375
0.359375 -0.359375
0.375 -0.375
0.390625 -0.390625
0.40625 -0.40625
0.421875 -0.421875
0.4375 -0.4375
0.453125 -0.453125
0.46875 -0.46875
0.484375 -0.484375
0.5 -0.5
0.515625 -0.515625
0.53125 -0.53125
0.546875 -0.546875
0.5625 -0.5625
0.578125 -0.578125
0.59375 -0.59375
0.609375 -0.609375
0.625 -0.625
0.640625 -0.640625
0.65625 -0.65625
0.671875 -0.671875
0.6875 -0.6875
0.703125 -0.703125
0.71875 -0.71875
0.734375 -0.734375
0.75 -0.75
0.765625 -0.765625
0.78125 -0.78125
0.796875 -0.796875
0.8125 -0.8125
0.828125 -0.828125
0.84375 -0.84375
0.859375 -0.859375
0.875 -0.875
0.890625 -0.890625
0.90625 -0.90625
0.921875 -0.921875
0.9375 -0.9375
0.953125 -0.953125
0.96875 -0.96875
0.984375 -0.984375
1 -1
1.015625 -1
1.03125 -1
1.046875 -1
1.0625 -1
1.078125 -1
1.09375 -1
1.109375 -1
1.125 -1
1.140625 -1
1.15625 -1
1.171875 -1
1.1875 -1
1.203125 -1
1.21875 -1
1.234375 -1
1.25 -1
1.265625 -1
1.28125 -1
1.296875 -1
1.3125 -1
1.328125 -1
1.34375 -1
1.359375 -1
1.375 -1
1.390625 -1
1.40625 -1
1.421875 -1
1.4375 -1
1.453125 -1
1.46875 -1
1.484375 -1
1.5 -1
1.515625 -1
1.53125 -1
1.546875 -1
1.5625 -1
1.578125 -1
1.59375 -1
1.609375 -1
1.625 -1
1.640625 -1
1.65625 -1
1.671875 -1
1.6875 -1
1.703125 -1
1.71875 -1
1.734375 -1
1.75 -1
1.765625 -1
1.78125 -1
1.796875 -1
1.8125 -1
1.828125 -1
1.84375 -1
1.859375 -1
1.875 -1
1.890625 -1
1.90625 -1
1.921875 -1
1.9375 -1
1.953125 -1
1.96875 -1
1.984375 -1
2 -1
2.015625 -1
2.03125 -1
2.046875 -1
2.0625 -1
2.078125 -1
2.09375 -1
2.109375 -1
2.125 -1
2.140625 -1
2.15625 -1
2.171875 -1
2.1875 -1
2.203125 -1
2.21875 -1
2.234375 -1
2.25 -1
2.265625 -1
2.28125 -1
2.296875 -1
2.3125 -1
2.328125 -1
2.34375 -1
2.359375 -1
2.375 -1
2.390625 -1
2.40625 -1
2.421875 -1
2.4375 -1
2.453125 -1
2.46875 -1
2.484375 -1
2.5 -1
2.515625 -1
2.53125 -1
2.546875 -1
2.5625 -1
2.578125 -1
2.59375 -1
2.609375 -1
2.625 -1
2.640625 -1
2.65625 -1
2.671875 -1
2.6875 -1
2.703125 -1
2.71875 -1
2.734375 -1
2.75 -1
2.765625 -1
2.78125 -1
2.796875 -1
2.8125 -1
2.828125 -1
2.84375 -1
2.859375 -1
2.875 -1
2.890625 -1
2.90625 -1
2.921875 -1
2.9375 -1
2.953125 -1
2.96875 -1
2.984375 -1
3 -1
};
\node at (axis cs:2.115,-1.2)[
  scale=1,
  anchor=base west,
  text=black,
  rotate=0.0
]{ $E(x)$};
\end{axis}

\end{tikzpicture}
    }
    
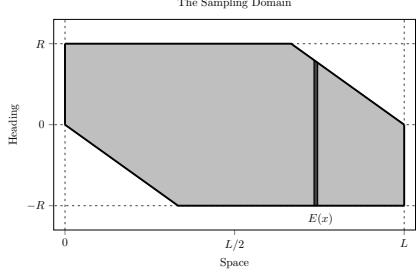
\captionof{figure}{A plot of the naive sensing domain $E_2(x)$ see
    \cref{Example:Naive}.
    }\label{Fig:DomainOfIntegration}
\end{minipage}\hfill
\begin{minipage}[t]{0.45\textwidth}
    \resizebox{0.95\textwidth}{!}{%
\begin{tikzpicture}

\begin{axis}[
title={The Sampling Domain},
xlabel={Space},
ylabel={Heading},
xmin=-0.1, xmax=3.1,
ymin=-1.3, ymax=1.3,
width=12.6cm,
height=8cm,
xtick={0,1.5,3},
xticklabels={$0$,$L/2$,$L$},
ytick={-1,0,1},
yticklabels={$-R$,$0$,$R$},
tick align=outside,
tick pos=left,
x grid style={lightgray!92.026143790849673!black},
y grid style={lightgray!92.026143790849673!black}
]
\path [draw=black, draw opacity=0.75, line width=0.83630769230769231pt, dash pattern=on 2pt off 3pt] (axis cs:0,-1.3)
--(axis cs:0,1.3);

\path [draw=black, draw opacity=0.75, line width=0.83630769230769231pt, dash pattern=on 2pt off 3pt] (axis cs:3,-1.3)
--(axis cs:3,1.3);

\path [draw=black, draw opacity=0.75, line width=0.83630769230769231pt, dash pattern=on 2pt off 3pt] (axis cs:-0.5,-1)
--(axis cs:3.5,-1);

\path [draw=black, draw opacity=0.75, line width=0.83630769230769231pt, dash pattern=on 2pt off 3pt] (axis cs:-0.5,0)
--(axis cs:3.5,0);

\path [draw=black, draw opacity=0.75, line width=0.83630769230769231pt, dash pattern=on 2pt off 3pt] (axis cs:-0.5,1)
--(axis cs:3.5,1);

\path [draw=lightgray, fill=lightgray] (axis cs:0,1)
--(axis cs:0,1)
--(axis cs:0.015625,0.96875)
--(axis cs:0.03125,0.9375)
--(axis cs:0.046875,0.90625)
--(axis cs:0.0625,0.875)
--(axis cs:0.078125,0.84375)
--(axis cs:0.09375,0.8125)
--(axis cs:0.109375,0.78125)
--(axis cs:0.125,0.75)
--(axis cs:0.140625,0.71875)
--(axis cs:0.15625,0.6875)
--(axis cs:0.171875,0.65625)
--(axis cs:0.1875,0.625)
--(axis cs:0.203125,0.59375)
--(axis cs:0.21875,0.5625)
--(axis cs:0.234375,0.53125)
--(axis cs:0.25,0.5)
--(axis cs:0.265625,0.46875)
--(axis cs:0.28125,0.4375)
--(axis cs:0.296875,0.40625)
--(axis cs:0.3125,0.375)
--(axis cs:0.328125,0.34375)
--(axis cs:0.34375,0.3125)
--(axis cs:0.359375,0.28125)
--(axis cs:0.375,0.25)
--(axis cs:0.390625,0.21875)
--(axis cs:0.40625,0.1875)
--(axis cs:0.421875,0.15625)
--(axis cs:0.4375,0.125)
--(axis cs:0.453125,0.09375)
--(axis cs:0.46875,0.0625)
--(axis cs:0.484375,0.03125)
--(axis cs:0.5,0)
--(axis cs:0.515625,-0.03125)
--(axis cs:0.53125,-0.0625)
--(axis cs:0.546875,-0.09375)
--(axis cs:0.5625,-0.125)
--(axis cs:0.578125,-0.15625)
--(axis cs:0.59375,-0.1875)
--(axis cs:0.609375,-0.21875)
--(axis cs:0.625,-0.25)
--(axis cs:0.640625,-0.28125)
--(axis cs:0.65625,-0.3125)
--(axis cs:0.671875,-0.34375)
--(axis cs:0.6875,-0.375)
--(axis cs:0.703125,-0.40625)
--(axis cs:0.71875,-0.4375)
--(axis cs:0.734375,-0.46875)
--(axis cs:0.75,-0.5)
--(axis cs:0.765625,-0.53125)
--(axis cs:0.78125,-0.5625)
--(axis cs:0.796875,-0.59375)
--(axis cs:0.8125,-0.625)
--(axis cs:0.828125,-0.65625)
--(axis cs:0.84375,-0.6875)
--(axis cs:0.859375,-0.71875)
--(axis cs:0.875,-0.75)
--(axis cs:0.890625,-0.78125)
--(axis cs:0.90625,-0.8125)
--(axis cs:0.921875,-0.84375)
--(axis cs:0.9375,-0.875)
--(axis cs:0.953125,-0.90625)
--(axis cs:0.96875,-0.9375)
--(axis cs:0.984375,-0.96875)
--(axis cs:1,-1)
--(axis cs:1.015625,-1)
--(axis cs:1.03125,-1)
--(axis cs:1.046875,-1)
--(axis cs:1.0625,-1)
--(axis cs:1.078125,-1)
--(axis cs:1.09375,-1)
--(axis cs:1.109375,-1)
--(axis cs:1.125,-1)
--(axis cs:1.140625,-1)
--(axis cs:1.15625,-1)
--(axis cs:1.171875,-1)
--(axis cs:1.1875,-1)
--(axis cs:1.203125,-1)
--(axis cs:1.21875,-1)
--(axis cs:1.234375,-1)
--(axis cs:1.25,-1)
--(axis cs:1.265625,-1)
--(axis cs:1.28125,-1)
--(axis cs:1.296875,-1)
--(axis cs:1.3125,-1)
--(axis cs:1.328125,-1)
--(axis cs:1.34375,-1)
--(axis cs:1.359375,-1)
--(axis cs:1.375,-1)
--(axis cs:1.390625,-1)
--(axis cs:1.40625,-1)
--(axis cs:1.421875,-1)
--(axis cs:1.4375,-1)
--(axis cs:1.453125,-1)
--(axis cs:1.46875,-1)
--(axis cs:1.484375,-1)
--(axis cs:1.5,-1)
--(axis cs:1.515625,-1)
--(axis cs:1.53125,-1)
--(axis cs:1.546875,-1)
--(axis cs:1.5625,-1)
--(axis cs:1.578125,-1)
--(axis cs:1.59375,-1)
--(axis cs:1.609375,-1)
--(axis cs:1.625,-1)
--(axis cs:1.640625,-1)
--(axis cs:1.65625,-1)
--(axis cs:1.671875,-1)
--(axis cs:1.6875,-1)
--(axis cs:1.703125,-1)
--(axis cs:1.71875,-1)
--(axis cs:1.734375,-1)
--(axis cs:1.75,-1)
--(axis cs:1.765625,-1)
--(axis cs:1.78125,-1)
--(axis cs:1.796875,-1)
--(axis cs:1.8125,-1)
--(axis cs:1.828125,-1)
--(axis cs:1.84375,-1)
--(axis cs:1.859375,-1)
--(axis cs:1.875,-1)
--(axis cs:1.890625,-1)
--(axis cs:1.90625,-1)
--(axis cs:1.921875,-1)
--(axis cs:1.9375,-1)
--(axis cs:1.953125,-1)
--(axis cs:1.96875,-1)
--(axis cs:1.984375,-1)
--(axis cs:2,-1)
--(axis cs:2.015625,-1)
--(axis cs:2.03125,-1)
--(axis cs:2.046875,-1)
--(axis cs:2.0625,-1)
--(axis cs:2.078125,-1)
--(axis cs:2.09375,-1)
--(axis cs:2.109375,-1)
--(axis cs:2.125,-1)
--(axis cs:2.140625,-1)
--(axis cs:2.15625,-1)
--(axis cs:2.171875,-1)
--(axis cs:2.1875,-1)
--(axis cs:2.203125,-1)
--(axis cs:2.21875,-1)
--(axis cs:2.234375,-1)
--(axis cs:2.25,-1)
--(axis cs:2.265625,-1)
--(axis cs:2.28125,-1)
--(axis cs:2.296875,-1)
--(axis cs:2.3125,-1)
--(axis cs:2.328125,-1)
--(axis cs:2.34375,-1)
--(axis cs:2.359375,-1)
--(axis cs:2.375,-1)
--(axis cs:2.390625,-1)
--(axis cs:2.40625,-1)
--(axis cs:2.421875,-1)
--(axis cs:2.4375,-1)
--(axis cs:2.453125,-1)
--(axis cs:2.46875,-1)
--(axis cs:2.484375,-1)
--(axis cs:2.5,-1)
--(axis cs:2.515625,-1)
--(axis cs:2.53125,-1)
--(axis cs:2.546875,-1)
--(axis cs:2.5625,-1)
--(axis cs:2.578125,-1)
--(axis cs:2.59375,-1)
--(axis cs:2.609375,-1)
--(axis cs:2.625,-1)
--(axis cs:2.640625,-1)
--(axis cs:2.65625,-1)
--(axis cs:2.671875,-1)
--(axis cs:2.6875,-1)
--(axis cs:2.703125,-1)
--(axis cs:2.71875,-1)
--(axis cs:2.734375,-1)
--(axis cs:2.75,-1)
--(axis cs:2.765625,-1)
--(axis cs:2.78125,-1)
--(axis cs:2.796875,-1)
--(axis cs:2.8125,-1)
--(axis cs:2.828125,-1)
--(axis cs:2.84375,-1)
--(axis cs:2.859375,-1)
--(axis cs:2.875,-1)
--(axis cs:2.890625,-1)
--(axis cs:2.90625,-1)
--(axis cs:2.921875,-1)
--(axis cs:2.9375,-1)
--(axis cs:2.953125,-1)
--(axis cs:2.96875,-1)
--(axis cs:2.984375,-1)
--(axis cs:3,-1)
--(axis cs:3,-1)
--(axis cs:3,-1)
--(axis cs:2.984375,-0.96875)
--(axis cs:2.96875,-0.9375)
--(axis cs:2.953125,-0.90625)
--(axis cs:2.9375,-0.875)
--(axis cs:2.921875,-0.84375)
--(axis cs:2.90625,-0.8125)
--(axis cs:2.890625,-0.78125)
--(axis cs:2.875,-0.75)
--(axis cs:2.859375,-0.71875)
--(axis cs:2.84375,-0.6875)
--(axis cs:2.828125,-0.65625)
--(axis cs:2.8125,-0.625)
--(axis cs:2.796875,-0.59375)
--(axis cs:2.78125,-0.5625)
--(axis cs:2.765625,-0.53125)
--(axis cs:2.75,-0.5)
--(axis cs:2.734375,-0.46875)
--(axis cs:2.71875,-0.4375)
--(axis cs:2.703125,-0.40625)
--(axis cs:2.6875,-0.375)
--(axis cs:2.671875,-0.34375)
--(axis cs:2.65625,-0.3125)
--(axis cs:2.640625,-0.28125)
--(axis cs:2.625,-0.25)
--(axis cs:2.609375,-0.21875)
--(axis cs:2.59375,-0.1875)
--(axis cs:2.578125,-0.15625)
--(axis cs:2.5625,-0.125)
--(axis cs:2.546875,-0.09375)
--(axis cs:2.53125,-0.0625)
--(axis cs:2.515625,-0.03125)
--(axis cs:2.5,0)
--(axis cs:2.484375,0.03125)
--(axis cs:2.46875,0.0625)
--(axis cs:2.453125,0.09375)
--(axis cs:2.4375,0.125)
--(axis cs:2.421875,0.15625)
--(axis cs:2.40625,0.1875)
--(axis cs:2.390625,0.21875)
--(axis cs:2.375,0.25)
--(axis cs:2.359375,0.28125)
--(axis cs:2.34375,0.3125)
--(axis cs:2.328125,0.34375)
--(axis cs:2.3125,0.375)
--(axis cs:2.296875,0.40625)
--(axis cs:2.28125,0.4375)
--(axis cs:2.265625,0.46875)
--(axis cs:2.25,0.5)
--(axis cs:2.234375,0.53125)
--(axis cs:2.21875,0.5625)
--(axis cs:2.203125,0.59375)
--(axis cs:2.1875,0.625)
--(axis cs:2.171875,0.65625)
--(axis cs:2.15625,0.6875)
--(axis cs:2.140625,0.71875)
--(axis cs:2.125,0.75)
--(axis cs:2.109375,0.78125)
--(axis cs:2.09375,0.8125)
--(axis cs:2.078125,0.84375)
--(axis cs:2.0625,0.875)
--(axis cs:2.046875,0.90625)
--(axis cs:2.03125,0.9375)
--(axis cs:2.015625,0.96875)
--(axis cs:2,1)
--(axis cs:1.984375,1)
--(axis cs:1.96875,1)
--(axis cs:1.953125,1)
--(axis cs:1.9375,1)
--(axis cs:1.921875,1)
--(axis cs:1.90625,1)
--(axis cs:1.890625,1)
--(axis cs:1.875,1)
--(axis cs:1.859375,1)
--(axis cs:1.84375,1)
--(axis cs:1.828125,1)
--(axis cs:1.8125,1)
--(axis cs:1.796875,1)
--(axis cs:1.78125,1)
--(axis cs:1.765625,1)
--(axis cs:1.75,1)
--(axis cs:1.734375,1)
--(axis cs:1.71875,1)
--(axis cs:1.703125,1)
--(axis cs:1.6875,1)
--(axis cs:1.671875,1)
--(axis cs:1.65625,1)
--(axis cs:1.640625,1)
--(axis cs:1.625,1)
--(axis cs:1.609375,1)
--(axis cs:1.59375,1)
--(axis cs:1.578125,1)
--(axis cs:1.5625,1)
--(axis cs:1.546875,1)
--(axis cs:1.53125,1)
--(axis cs:1.515625,1)
--(axis cs:1.5,1)
--(axis cs:1.484375,1)
--(axis cs:1.46875,1)
--(axis cs:1.453125,1)
--(axis cs:1.4375,1)
--(axis cs:1.421875,1)
--(axis cs:1.40625,1)
--(axis cs:1.390625,1)
--(axis cs:1.375,1)
--(axis cs:1.359375,1)
--(axis cs:1.34375,1)
--(axis cs:1.328125,1)
--(axis cs:1.3125,1)
--(axis cs:1.296875,1)
--(axis cs:1.28125,1)
--(axis cs:1.265625,1)
--(axis cs:1.25,1)
--(axis cs:1.234375,1)
--(axis cs:1.21875,1)
--(axis cs:1.203125,1)
--(axis cs:1.1875,1)
--(axis cs:1.171875,1)
--(axis cs:1.15625,1)
--(axis cs:1.140625,1)
--(axis cs:1.125,1)
--(axis cs:1.109375,1)
--(axis cs:1.09375,1)
--(axis cs:1.078125,1)
--(axis cs:1.0625,1)
--(axis cs:1.046875,1)
--(axis cs:1.03125,1)
--(axis cs:1.015625,1)
--(axis cs:1,1)
--(axis cs:0.984375,1)
--(axis cs:0.96875,1)
--(axis cs:0.953125,1)
--(axis cs:0.9375,1)
--(axis cs:0.921875,1)
--(axis cs:0.90625,1)
--(axis cs:0.890625,1)
--(axis cs:0.875,1)
--(axis cs:0.859375,1)
--(axis cs:0.84375,1)
--(axis cs:0.828125,1)
--(axis cs:0.8125,1)
--(axis cs:0.796875,1)
--(axis cs:0.78125,1)
--(axis cs:0.765625,1)
--(axis cs:0.75,1)
--(axis cs:0.734375,1)
--(axis cs:0.71875,1)
--(axis cs:0.703125,1)
--(axis cs:0.6875,1)
--(axis cs:0.671875,1)
--(axis cs:0.65625,1)
--(axis cs:0.640625,1)
--(axis cs:0.625,1)
--(axis cs:0.609375,1)
--(axis cs:0.59375,1)
--(axis cs:0.578125,1)
--(axis cs:0.5625,1)
--(axis cs:0.546875,1)
--(axis cs:0.53125,1)
--(axis cs:0.515625,1)
--(axis cs:0.5,1)
--(axis cs:0.484375,1)
--(axis cs:0.46875,1)
--(axis cs:0.453125,1)
--(axis cs:0.4375,1)
--(axis cs:0.421875,1)
--(axis cs:0.40625,1)
--(axis cs:0.390625,1)
--(axis cs:0.375,1)
--(axis cs:0.359375,1)
--(axis cs:0.34375,1)
--(axis cs:0.328125,1)
--(axis cs:0.3125,1)
--(axis cs:0.296875,1)
--(axis cs:0.28125,1)
--(axis cs:0.265625,1)
--(axis cs:0.25,1)
--(axis cs:0.234375,1)
--(axis cs:0.21875,1)
--(axis cs:0.203125,1)
--(axis cs:0.1875,1)
--(axis cs:0.171875,1)
--(axis cs:0.15625,1)
--(axis cs:0.140625,1)
--(axis cs:0.125,1)
--(axis cs:0.109375,1)
--(axis cs:0.09375,1)
--(axis cs:0.078125,1)
--(axis cs:0.0625,1)
--(axis cs:0.046875,1)
--(axis cs:0.03125,1)
--(axis cs:0.015625,1)
--(axis cs:0,1)
--cycle;

\path [draw=darkgray, fill=darkgray] (axis cs:2.203125,0.59375)
--(axis cs:2.203125,-1)
--(axis cs:2.21875,-1)
--(axis cs:2.234375,-1)
--(axis cs:2.234375,0.53125)
--(axis cs:2.234375,0.53125)
--(axis cs:2.21875,0.5625)
--(axis cs:2.203125,0.59375)
--cycle;

\path [draw=black, line width=0.83630769230769231pt] (axis cs:2.203125,-1)
--(axis cs:2.203125,0.59375);

\path [draw=black, line width=0.83630769230769231pt] (axis cs:2.234375,-1)
--(axis cs:2.234375,0.53125);

\addplot [line width=1.6726153846153846pt, black, forget plot]
table {%
0 1
0.015625 1
0.03125 1
0.046875 1
0.0625 1
0.078125 1
0.09375 1
0.109375 1
0.125 1
0.140625 1
0.15625 1
0.171875 1
0.1875 1
0.203125 1
0.21875 1
0.234375 1
0.25 1
0.265625 1
0.28125 1
0.296875 1
0.3125 1
0.328125 1
0.34375 1
0.359375 1
0.375 1
0.390625 1
0.40625 1
0.421875 1
0.4375 1
0.453125 1
0.46875 1
0.484375 1
0.5 1
0.515625 1
0.53125 1
0.546875 1
0.5625 1
0.578125 1
0.59375 1
0.609375 1
0.625 1
0.640625 1
0.65625 1
0.671875 1
0.6875 1
0.703125 1
0.71875 1
0.734375 1
0.75 1
0.765625 1
0.78125 1
0.796875 1
0.8125 1
0.828125 1
0.84375 1
0.859375 1
0.875 1
0.890625 1
0.90625 1
0.921875 1
0.9375 1
0.953125 1
0.96875 1
0.984375 1
1 1
1.015625 1
1.03125 1
1.046875 1
1.0625 1
1.078125 1
1.09375 1
1.109375 1
1.125 1
1.140625 1
1.15625 1
1.171875 1
1.1875 1
1.203125 1
1.21875 1
1.234375 1
1.25 1
1.265625 1
1.28125 1
1.296875 1
1.3125 1
1.328125 1
1.34375 1
1.359375 1
1.375 1
1.390625 1
1.40625 1
1.421875 1
1.4375 1
1.453125 1
1.46875 1
1.484375 1
1.5 1
1.515625 1
1.53125 1
1.546875 1
1.5625 1
1.578125 1
1.59375 1
1.609375 1
1.625 1
1.640625 1
1.65625 1
1.671875 1
1.6875 1
1.703125 1
1.71875 1
1.734375 1
1.75 1
1.765625 1
1.78125 1
1.796875 1
1.8125 1
1.828125 1
1.84375 1
1.859375 1
1.875 1
1.890625 1
1.90625 1
1.921875 1
1.9375 1
1.953125 1
1.96875 1
1.984375 1
2 1
2.015625 0.96875
2.03125 0.9375
2.046875 0.90625
2.0625 0.875
2.078125 0.84375
2.09375 0.8125
2.109375 0.78125
2.125 0.75
2.140625 0.71875
2.15625 0.6875
2.171875 0.65625
2.1875 0.625
2.203125 0.59375
2.21875 0.5625
2.234375 0.53125
2.25 0.5
2.265625 0.46875
2.28125 0.4375
2.296875 0.40625
2.3125 0.375
2.328125 0.34375
2.34375 0.3125
2.359375 0.28125
2.375 0.25
2.390625 0.21875
2.40625 0.1875
2.421875 0.15625
2.4375 0.125
2.453125 0.09375
2.46875 0.0625
2.484375 0.03125
2.5 0
2.515625 -0.03125
2.53125 -0.0625
2.546875 -0.09375
2.5625 -0.125
2.578125 -0.15625
2.59375 -0.1875
2.609375 -0.21875
2.625 -0.25
2.640625 -0.28125
2.65625 -0.3125
2.671875 -0.34375
2.6875 -0.375
2.703125 -0.40625
2.71875 -0.4375
2.734375 -0.46875
2.75 -0.5
2.765625 -0.53125
2.78125 -0.5625
2.796875 -0.59375
2.8125 -0.625
2.828125 -0.65625
2.84375 -0.6875
2.859375 -0.71875
2.875 -0.75
2.890625 -0.78125
2.90625 -0.8125
2.921875 -0.84375
2.9375 -0.875
2.953125 -0.90625
2.96875 -0.9375
2.984375 -0.96875
3 -1
};
\addplot [line width=1.6726153846153846pt, black, forget plot]
table {%
0 1
0.015625 0.96875
0.03125 0.9375
0.046875 0.90625
0.0625 0.875
0.078125 0.84375
0.09375 0.8125
0.109375 0.78125
0.125 0.75
0.140625 0.71875
0.15625 0.6875
0.171875 0.65625
0.1875 0.625
0.203125 0.59375
0.21875 0.5625
0.234375 0.53125
0.25 0.5
0.265625 0.46875
0.28125 0.4375
0.296875 0.40625
0.3125 0.375
0.328125 0.34375
0.34375 0.3125
0.359375 0.28125
0.375 0.25
0.390625 0.21875
0.40625 0.1875
0.421875 0.15625
0.4375 0.125
0.453125 0.09375
0.46875 0.0625
0.484375 0.03125
0.5 0
0.515625 -0.03125
0.53125 -0.0625
0.546875 -0.09375
0.5625 -0.125
0.578125 -0.15625
0.59375 -0.1875
0.609375 -0.21875
0.625 -0.25
0.640625 -0.28125
0.65625 -0.3125
0.671875 -0.34375
0.6875 -0.375
0.703125 -0.40625
0.71875 -0.4375
0.734375 -0.46875
0.75 -0.5
0.765625 -0.53125
0.78125 -0.5625
0.796875 -0.59375
0.8125 -0.625
0.828125 -0.65625
0.84375 -0.6875
0.859375 -0.71875
0.875 -0.75
0.890625 -0.78125
0.90625 -0.8125
0.921875 -0.84375
0.9375 -0.875
0.953125 -0.90625
0.96875 -0.9375
0.984375 -0.96875
1 -1
1.015625 -1
1.03125 -1
1.046875 -1
1.0625 -1
1.078125 -1
1.09375 -1
1.109375 -1
1.125 -1
1.140625 -1
1.15625 -1
1.171875 -1
1.1875 -1
1.203125 -1
1.21875 -1
1.234375 -1
1.25 -1
1.265625 -1
1.28125 -1
1.296875 -1
1.3125 -1
1.328125 -1
1.34375 -1
1.359375 -1
1.375 -1
1.390625 -1
1.40625 -1
1.421875 -1
1.4375 -1
1.453125 -1
1.46875 -1
1.484375 -1
1.5 -1
1.515625 -1
1.53125 -1
1.546875 -1
1.5625 -1
1.578125 -1
1.59375 -1
1.609375 -1
1.625 -1
1.640625 -1
1.65625 -1
1.671875 -1
1.6875 -1
1.703125 -1
1.71875 -1
1.734375 -1
1.75 -1
1.765625 -1
1.78125 -1
1.796875 -1
1.8125 -1
1.828125 -1
1.84375 -1
1.859375 -1
1.875 -1
1.890625 -1
1.90625 -1
1.921875 -1
1.9375 -1
1.953125 -1
1.96875 -1
1.984375 -1
2 -1
2.015625 -1
2.03125 -1
2.046875 -1
2.0625 -1
2.078125 -1
2.09375 -1
2.109375 -1
2.125 -1
2.140625 -1
2.15625 -1
2.171875 -1
2.1875 -1
2.203125 -1
2.21875 -1
2.234375 -1
2.25 -1
2.265625 -1
2.28125 -1
2.296875 -1
2.3125 -1
2.328125 -1
2.34375 -1
2.359375 -1
2.375 -1
2.390625 -1
2.40625 -1
2.421875 -1
2.4375 -1
2.453125 -1
2.46875 -1
2.484375 -1
2.5 -1
2.515625 -1
2.53125 -1
2.546875 -1
2.5625 -1
2.578125 -1
2.59375 -1
2.609375 -1
2.625 -1
2.640625 -1
2.65625 -1
2.671875 -1
2.6875 -1
2.703125 -1
2.71875 -1
2.734375 -1
2.75 -1
2.765625 -1
2.78125 -1
2.796875 -1
2.8125 -1
2.828125 -1
2.84375 -1
2.859375 -1
2.875 -1
2.890625 -1
2.90625 -1
2.921875 -1
2.9375 -1
2.953125 -1
2.96875 -1
2.984375 -1
3 -1
};
\node at (axis cs:2.115,-1.2)[
  scale=1.0,
  anchor=base west,
  text=black,
  rotate=0.0
]{ $E(x)$};
\end{axis}

\end{tikzpicture}
    }
    
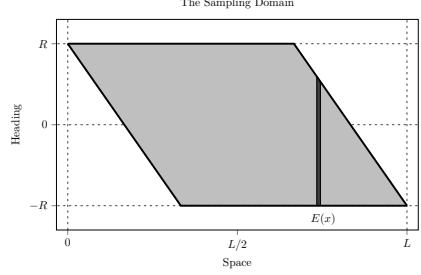
\captionof{figure}{A plot of the no flux sensing domain $E_3(x)$ see
    \cref{Example:NoFlux}.
    }\label{Fig:NoFluxDomainOfIntegration}
\end{minipage}

\begin{example}[Adhesive / Repulsive]\label{Example:WeaklyAdhesive}
    The framework developed here can be used to explicitly model adhesion or
    repulsion by the domain boundary. For that we assume that the interaction
    force with the boundary is proportional to the extent of cell protrusions
    that attach to the boundary, which corresponds to the amount of cell
    protrusion that would reach out of the domain if there was no boundary
    (see \cref{fig:bioboundary1}~(B)). For
    example at $x\in (0,R)$. If the cell extends to $x-R$, then the interval
    $[x-R,0)$ is outside of the domain. We assume that instead of leaving the
    domain, the protrusion interacts with the boundary, given boundary adhesion
    terms of the form
    \begin{eqnarray*}
        a^0(x) \coloneqq \beta^0 \int_{-R}^{-x} \Omega(r) \dd r, &\qquad& x\in[0,R)\\
        a^L(x) \coloneqq \beta^L\int_{L-x}^R \Omega(r) \dd r, && x\in(L-R,L]
    \end{eqnarray*}
    where $\beta^0$ and $\beta^L$ are constants of proportionality.
    $\beta^0,\beta^L>0$ describes boundary adhesion, while $\beta^0, \beta^L<0$
    describes boundary repulsion.

    In this case we define the adhesion operator as linear combination of all
    relevant adhesive effects. Using indicator functions $\chi_A(r)$ we can
    write
    \begin{align}\label{Fadhesive}
        \K[u](x,t) &= \int_{E(x)} H(u(x+r, t)) \Omega(r) \dd r\nonumber\\
        &            + \beta^0 \chi_{[0,R]}(x)  \int_{-R}^{-x} \Omega(r) \dd r
            + \beta^L \chi_{[L-R,L]}(x) \int_{L-x}^{R} \Omega(r) \dd r \nonumber\\
        &= \int_{-R}^R\left(\chi_{E(x)}H(u(x+r,t))+\beta^0\chi_{[-R,-x]}(x) +
        \beta^L\chi_{[L-x,R]} \right)\Omega(r) \dd r,
    \end{align}
    where we omitted the $r$-dependence in the indicator functions for brevity.
    Here $E(x)$ is any suitable sampling domain as defined in
    Definition~\ref{def:slice}. Further we note that whenever
    \[
        \beta^0 = \frac{1}{2} \int_{E(0)} H(u(r, t))\Omega(r) \dd r,
    \]
    a similar expression can be found for $\beta^L$, then $\K[u]$ satisfies condition~\eqref{nofluxK}.
\end{example}

\paragraph{Example Summary}
To combine all different possible examples, we define a general integral operator as
\begin{equation}\label{generalK}
    \K[u](x,t) = \int_{-R}^R F(u(x+r,t), x, r) \Omega(r) \dd r,
\end{equation}
with
\begin{eqnarray}\label{Fexamples}
\small
\mbox{periodic} &\quad& F(u,x,r) = H(u(x+r,t)) \nonumber\\
\mbox{naive case} & & F(u,x,r) =\chi_{E_2(x)} H(u(x+r, t))\nonumber\\
\mbox{non-flux} & & F(u,x,r) =\chi_{E_3(x)} H(u(x+r, t))\\
\mbox{general case} & & F(u,x,r) =\chi_{E(x)} H(u(x+r, t)), \quad E(x) \mbox{ is suitable}\nonumber\\
\mbox{bdy interac.} & & F(u,x,r) =\chi_{E(x)} H(u(x+r, t))+\beta^0\chi_{[-R,-x]} + \beta^L \chi_{[L-x,R]} \nonumber\\
&& E(x) \mbox{ is suitable}, \quad \beta^0, \beta^L \mbox{ are constants},\nonumber
\end{eqnarray}
where ``bdy interac.'' stands for ``adhesive or repulsive interaction with the boundary''.
We will summarize general assumptions on $F$ in the next section.

\section{Local Existence and Uniqueness}\label{S:existence}

We consider non-local adhesion models on a one-dimensional bounded domain $[0,L]$ with independent
no-flux boundary condition:
\begin{equation}\label{mainmodel}
\left\{\begin{array}{rcl}
    u_t(x,t) &=& Du_{xx}(x,t) - \alpha \left(u(x,t) \K[u](x, t) \right)_x \\[1ex]
    \K[u](x, t) &=& \int_{-R}^R F(u(x+r,t),x,r)\Omega(r) \dd r\\[1ex]
    u(x,0) &=& u_0(x) \geq 0 \\[1ex]
    0 &=& u_x(0, t) = u_x(L, t) \\[1ex]
    &\K[u](x)&\ \mbox{satisfies condition~\eqref{nofluxK}}.
\end{array} \right.
\end{equation}
We introduce the function space
\[
    {\cal Y} \coloneqq \left\{ u \in H^1[0,L]\colon \int_{0}^{L} u(x) \dd x = m_0\right\},
\]
where $m_0 = \int_{0}^{L} u_0(x) \dd x$. We recall that the function space
${\cal Y}$ can be identified with the quotient space $H^1 / \mathbb{R}$. We then pick the
solution of equation~\eqref{mainmodel} to be the representative with mass $m_0$.
From~\cite{necas2011} we recall that this quotient space is Hilbert, and that its norm
$|u|_{H^1 / \mathbb{R}}$ is equivalent to the norm
\[
    ||u|| \coloneqq \left( \int_{0}^{L} u_x^2 \dd x \right)^{1/2}.
\]

We make the following general assumptions:
\begin{enumerate}[label=\textbf{(A\arabic*)},ref=(A\arabic*), leftmargin=*,labelindent=\parindent]
\item\label{Assumption:1} $u_0 \in {\cal Y},\ {\cal X} = \C^0\lb[0,T],\ {\cal Y}\cap L^\infty(0,L)\rb,\ T>0$.
\item\label{Assumption:2} $\Omega(r) = \frac{r}{|r|} \omega(r),\ \omega(r)=\omega(-r),\ \omega(r)\geq 0,\ \omega(R)=0,\ R>0. $
\item\label{Assumption:3} $V=[-R,R],\ \omega\in L^1(V)\cap L^\infty(V),\  \|\omega\|_{L^1[0,R]} = \frac{1}{2}$.
\item\label{Assumption:4} For each $x\in[0,L],\ r\in[-R,R]$ the kernel $F(u,x,r)$ is linearly
    bounded in $u$ and differentiable in $u$ with uniformly bounded and Lipschitz continuous derivative:
    \[
        |F(u,x,r)|\leq k_1 (1+|u|), \qquad \left|\frac{\partial }{\partial u} F(u,x,r)\right|\leq k_2.
    \]
    \[ \left|\frac{\partial}{\partial u} F(u,x,r) - \frac{\partial}{\partial u}F(v,x,r) \right|\leq k_3 |u-v| .\]
\item\label{Assumption:5} $F(u,x,r)$ is piecewise continuous as a function of $r$.
\item\label{Assumption:6} $x\mapsto\int_V F \; \Omega(r) \dd r $  is differentiable in $x$ with a
    bounded derivative. There are two constants $k_4, k_5>0$ such that
    \begin{eqnarray*}
        \left| \frac{\partial}{\partial x} \int_V F(u(x+r,t),x,r) \Omega(r) \dd r \right|_2
            &\leq& k_4\lb1+|u(\cdot,t)|_\infty\rb \\
            &\leq& k_5\lb1+|u(\cdot,t)|_{H^1}\rb,
    \end{eqnarray*}
for all $u\in\cX, t>0$.
\end{enumerate}

\begin{lemma} Assume~\ref{Assumption:1}--\ref{Assumption:3}. Further assume that
\begin{description}
    \item {\bf (A4')} $H(u)$ is linearly bounded with uniform bounded and Lipschitz continuous  derivative.
\end{description}
Then all of the above examples (\ref{Fexamples}) satisfy assumptions~\ref{Assumption:1}--\ref{Assumption:6}.
\end{lemma}

\begin{proof}
The $u$-dependence in the examples~(\ref{Fexamples}) enters only
through $H(u)$. Hence assumption (A4') immediately implies
assumption~\ref{Assumption:4}.
Since $u$ is continuous and $H$ is continuous and the indicator functions are
piecewise continuous, then also $r\mapsto F(u,x,r)$ is piecewise continuous,
i.e.~\ref{Assumption:5}. The critical condition to show is
assumption~\ref{Assumption:6}. For this we consider the case of adhesive and repulsive
boundary conditions, as this proof
also includes the proof of~\ref{Assumption:6} for the other examples.
We have
\[
    F(u,x,r) =\chi_{E(x)} H(u(x+r, t))+\beta^0\chi_{[-R,-x]} + \beta^L \chi_{[L-x,R]}.
\]
Since $E(x)=\{r\in[-R,R]: f_1(x) \leq r\leq f_2(x) \}$ is a suitable slice, we
can compute the distributional derivative of $F$. We divide this into several
steps. Differentiating the integral term, we find
\begin{equation}\label{Eqn:DerivativeNonLocalTerm}
    \frac{\partial}{\partial x} \int_{V} F(x, u, r) \Omega(r) \dd r =
        \int_{V} \left[ F_u(u, x, r) u_x + F_x(u, x, r) \right] \Omega(r) \dd r.
\end{equation}
We use assumption~\ref{Assumption:4} to estimate the first term
\[
    \left| \int_V F_u(u,x,r) u_x(x+r,t) \Omega(r) \dd r \right|_2
        \leq k_2 |\Omega|_\infty |u_x|_2.
\]
The second term is more delicate. First we compute the distributional derivative
$F_x(u, x, r)$ for $r\in[-R,R]$ and $x\in [0,L]$:
\begin{align*}
    F_x(u,x,r)  &=
        H(u) \left[\mathcal{H}(r - f_1(x))\delta(f_2 - r) f_2'(x) -
                         \mathcal{H}(f_2 - r)\delta(r - f_1)f_1'(x)\right] +\\
        &\qquad \chi_{E(x)}(r) \frac{\partial H}{\partial x}
        + \beta^L\delta(r-L+x) - \beta^0 \delta(-x -r),
\end{align*}
where $\mathcal{H}$ is the heaviside function. We note that $H_x = 0$ (since we are only taking
the partial with respect to $x$ now). Integrating this term with weight $\Omega(r)$ over $V=[-R,R]$, and noting that $x\in[0,L]$, we get
\begin{align*}
   & \int_{V} F_x(u,x,r) \Omega(r) \dd r\\
        &= H(u(x + f_2(x)) f_2'(x) \Omega(f_2(x))
        - H(u(x + f_1(x))) f_1'(x) \Omega(f_1(x))  \\
        &\qquad- \beta^0 \chi_{[0, R)}(x) \Omega(-x)
        + \beta^L \chi_{(L-x, L]}(x) \Omega(L - x).
\end{align*}
Notice that all terms in the above expression only arise for $x$ close to the boundaries.
The terms involving $\beta^{0, L}$ are multiplied by the indicator functions of the
boundary region, while the other two terms are zero outside the boundary region, since
$f_{1,2}'(x) = 0$ (see Definition~\ref{def:slice}). Using this term we can estimate the
second term in equation~\eqref{Eqn:DerivativeNonLocalTerm} by
\[
    \left|\int_V F_x(u,x,r)\Omega(r) \dd r \right|_2
        \leq \Bigl(2 k_1 D (1+|u|_\infty) +|\beta^0| +|\beta^L|\Bigr) |\Omega|_\infty,
\]
where $D\coloneqq \max\left(|{f_1'}|_{\infty}, |{f_2'}|_{\infty}\right)$.
Together we find two constants $k_4, k_5 >0$ such that
\begin{eqnarray}\label{gradestimate}
\left|
\frac{\dd}{\dd x} \int_V F(u(x+r,t), x,r) \Omega(r) \dd r \right|_2
&\leq& k_4 \Bigl(1+ |u|_\infty\Bigr)\\
&\leq& k_5 \Bigl(1+ |u|_{H^1} \Bigr),
\end{eqnarray}
where the last estimate follows from the Sobolev embedding.

\end{proof}

We denote the solution semigroup $S(t)$ of the heat equation with homogeneous no-flux boundary conditions
\[ \left\{
\begin{array}{rcl}
u_t &=& D u_{xx} \\
0 &=& u_x(0,t)=u_x(L,t) \end{array}.
\right.\]

And we can write the formal solution of~\eqref{mainmodel} as a mild solution
\begin{definition}
$u\in \cX$ is called a mild solution of~\eqref{mainmodel} if
\begin{equation}\label{mildsolution}
u(x,t) =  S(t) u_0 - \alpha \int_0^t S(t-s)\left( u \int_V F(u(x+r,s),x,r)\Omega(r) \dd r\right)_x \dd s.
\end{equation}
\end{definition}

\begin{theorem}\label{t:existence} Assume~\ref{Assumption:1}--\ref{Assumption:6}. For
$T>0$ small enough there exists a unique mild solution $u\in\cX$ of~\eqref{mainmodel}.
\end{theorem}
\begin{proof}
Using this definition we can define a map $Q:\cX\to \cX$, where given $v\in \cX$, $u=Qv$ denotes the function
\begin{equation}\label{iteration}
u(x,t) =  S(t) u_0 - \alpha \int_0^t S(t-s)\left( v \int_V F(v(x+r,s),x,r)\Omega(r) \dd r\right)_x \dd s.
\end{equation}
We will show that this map has a unique fixed point in $\cX$.
Assume $v\in \cX$. By the Sobolev embedding this implies that $v\in \C^0([0,T], \C^0([0,L]))$. \\

\begin{enumerate}[label=\textbf{Step~\arabic*:},ref=Step~\arabic*,
                  leftmargin=*,labelindent=\parindent, wide]

\item For given $M>2\max\{|u_0|_{H^1}, |u_0|_\infty\} $ let $B_M(0)\subset H^1[0,L]\cap L^\infty(0,L)$ denote the ball of radius
$M$ in $H^1\cap L^\infty$. Let $W = \C^0([0,T], B_M(0))$, then we
show that for $T>0$ small enough we have $Q:W\to W$.
In the following estimates we ignore the arguments of the functions and we write $v=v(x,t)$,
$F=F(v(x+r,s),x,s)$, and $\Omega = \Omega(r)$. The crucial term is the integral term
in equation~\eqref{iteration}
\[
    \left(v\int_V F\;\Omega \dd r \right)_x = v_x \int_V F\;\Omega \dd r + v\frac{\dd}{\dd x} \int_V F\;\Omega \dd r.
\]
Then
\begin{eqnarray}
    \left|\left(v\int_V F\; \Omega \dd r\right)_x \right|_2 &\leq&
        \left| v_x\int_V F\; \Omega(r) \dd r \right|_2+\left| v \frac{\dd}{\dd x}\int_V F\;
            \Omega(r) \dd r\right|_2 \nonumber \\
    &\leq& k_1 |v_x|_2 \left(1+|v|_\infty\right)\int_V|\Omega(r)|\dd r
        + k_5 |v|_\infty (1+ |v|_{H^1})\nonumber  \\
    &\leq& \kappa\left( 1+|v|_{H^1} \right) \left(1+|v|_\infty \right)\nonumber \\
    &\leq& \kappa (1+M)^2,\label{oneplusM}
\end{eqnarray}
with $\kappa >0$ and we used $\int_V |\Omega(r)| \dd r =1$.

Now the heat solution semigroup regularizes \cite{Amann1995}
\begin{equation}\label{regularity}
 S(t): L^2[0,L] \to H^1[0,L] \mbox{ with norm } C t^{-1/2}.
 \end{equation}
Hence
\begin{eqnarray*}
    \left|\int_0^t S(t-s) \left(v\int_V F\; \Omega \dd r\right)_x \dd s
    \right|_{H^1} &\leq& \kappa (1+M)^2
    \left|\int_0^t C (t-s)^{-1/2} \dd s \right|\\
    &=& 2 \kappa C (1+M)^2 \sqrt{t}.
\end{eqnarray*}
Then from~\eqref{iteration} and the choice of $M$ we find that
\[
    |u|_{H^1} \leq \frac{M}{2} +2 \kappa \alpha C(1+M)^2 \sqrt{t},
\]
and
\[
    \frac{M}{2} + 2 \kappa \alpha C (1+M)^2 \sqrt{t} < M
\]
for all
\[
    t< M^2 (4\kappa \alpha C (1+M)^2 )^{-2}.
\]

\item Now we show that $Q$ is a contraction on $W$ for small enough time.
Given $v_1, v_2\in W$, let $u_1= Qv_1$ and $u_2=Q v_2$ and we abbreviate
$F_1=F(v_1(x+r,t),x,r)$ and $F_2=F(v_2(x+r,t),x,r)$.  We estimate for the $H^1$-norm:

\begin{eqnarray*}
|u_1-u_2|_{H^1}
&\leq & \alpha\left|\int_0^t S(t-s)
    \left[ \left( v_1 \int_V F_1\; \Omega \dd r \right)_x -
    \left( v_2 \int_V F_2 \; \Omega \dd r \right)_x\right] \dd s\right|_{H^1} \\
    &\leq & \alpha\left| \int_0^t S(t-s) \left((v_1-v_2) \int_V F_1\; \Omega \dd r \right)_x \dd s\right|_{H^1} \\
    && + \alpha \left| \int_0^t S(t-s) \left( v_2\int_V (F_1-F_2) \Omega \dd r \right)_x \dd s \right|_{H^1}\\
    &\leq& \alpha \left| \int_0^t S(t-s)  (v_1-v_2)_x \int_V F_1 \Omega \dd r  \dd s \right|_{H^1}\\
    && + \alpha \left| \int_0^t S(t-s) (v_1-v_2)\frac{\dd}{\dd x}\int_V F_1 \Omega \dd r  \dd s \right|_{H^1}\\
    && + \alpha \left| \int_0^t S(t-s)  v_{2,x}\int_V (F_1-F_2) \Omega \dd r  \dd s \right|_{H^1}\\
    && +\alpha \left| \int_0^t S(t-s) v_2\int_V (F_{1u}-F_{2u}) v_{1x} \Omega \dd r \dd s\right|_{H^1}\\
    && + \alpha \left| \int_0^t S(t-s) v_2 \int_V F_{2u} (v_{1x}-v_{2x}) \dd r \dd s\right|_{H^1}\\
    &=:& I_1 + I_2 + I_3 + I_4 + I_5.
    \end{eqnarray*}
We use the previous bounds of $|v_1|_\cX, |v_2|_\cX\leq M$ and \eqref{oneplusM} to study each term separately. We also use the regularization of the heat equation semigroup \eqref{regularity} for all terms $I_k$. We obtain
\begin{eqnarray*}
I_1 &\leq & \alpha \sqrt{t} (1+M) |v_1-v_2|_\cX\\
I_2 &\leq & \alpha \sqrt{t} (1+M) k_5 |v_1-v_2|_\cX\\
I_3 &\leq & \alpha \sqrt{t} M k_2 |v_1-v_2|_\cX\\
I_4 &\leq & \alpha \sqrt{t} M^2 k_3 |v_1-v_2|_\cX\\
I_5 &\leq & \alpha \sqrt{t} M k_2 |v_1-v_2|_{\cX}.
\end{eqnarray*}
Which means that there is a constant $C>0$ such that
\[ |u_1-u_2|_{H^1} \leq C \sqrt{t} |v_1-v_2|_\cX.\]
Note that since in one-dimension we have that $H^1 \subset L^\infty$ we automatically
have the same estimate for the supremum norm. Together we find a constant $C>0$ such that
\[ |u_1-u_2|_\cX \leq C \sqrt{t}  |v_1-v_2|_\cX,\]
which is a contraction for $t$ small enough.

\item The map $Q$ is a continuous contraction on $B_M(0)$ for small enough times, hence there exists a unique short-time mild solution of (\ref{mildsolution}).

\end{enumerate}
\end{proof}

\section{Global Existence}\label{sec:global_existence}

\begin{lemma} Assume~\ref{Assumption:1}--\ref{Assumption:6} and let $u(x,t)$ denote the unique,
non-negative, mild solution of
\eqref{mainmodel} from Theorem~\ref{t:existence}. Then there is a constant $c_1 > 0$ such that
\begin{equation}\label{L2expo}
    |u(\cdot,t)|_2 \leq |u_0|_2 e^{c_1 t},
\end{equation}
for as long as the solution exists.
\end{lemma}
\begin{proof}
We multiply \eqref{mainmodel} by $u$ and integrate:
\begin{eqnarray*}
\frac{\dd}{\dd t}\int \frac{u^2}{2} \dd x &=& -D \int u_x^2 \dd x + \alpha\int u_x \left(u \int_V F(u,x,r) \Omega(r) \dd r\right) \dd x \\
&\leq & - D \int u_x^2 +\frac{\alpha\ep}{2} \int u_x^2 \dd x + \frac{\alpha}{2\ep} \int\left(u \int_V F(u,x,r) \Omega(r) \dd r\right)^2 \dd x\\
&\leq & \left(-D + \frac{\alpha\ep}{2}\right) \int u_x^2 \dd x +
    \frac{\alpha}{2\ep} \int\left[u\int_V k_1 (1+|u|) \Omega(r) \dd r \right]^2 \dd x \\
&\leq &   \left(-D + \frac{\alpha\ep}{2}\right) \int u_x^2 \dd x +
    \frac{\alpha}{2\ep} \int u^2 \Bigl[k_1 (2R + |u|_{1}) |\Omega|_\infty\Bigr]^2 \dd x \\
&\leq &  \left(-D + \frac{\alpha\ep}{2}\right) \int u_x^2 \dd x + \left[\frac{\alpha}{2\ep}(2 R + m_0) |\Omega|_\infty\right]^2 \int u^2 \dd x.\\
\end{eqnarray*}
We choose $\ep=2 D/\alpha$, such that the first term cancels and we obtain
\[
    \frac{\dd}{\dd t}\int\frac{u^2}{2} \dd x \leq \left[\frac{k_1 \alpha^2}{4D} (2 R + m_0) |\Omega|_\infty \right]^2 \int u^2 \dd x.
\]
Hence there is a constant $c_1>0$ such that~\eqref{L2expo} is satisfied.
\end{proof}

\begin{theorem}\label{Thm:GlobalExistence}
Assume~\ref{Assumption:1}--\ref{Assumption:6}
and let $u(x,t)$ denote the unique, non-negative, mild solution of \eqref{mainmodel} from
Theorem~\ref{t:existence}. Then the solution exists globally in time and there are
constants $c_2, c_3>0$ such that
\begin{equation}\label{H1expo}
    ||u(\cdot,t)|| \leq c_2(||u_0|| + t) e^{c_3 t}.
\end{equation}
\end{theorem}
\begin{proof}
We multiply (\ref{mainmodel}) by $u_{xx}$ and integrate:
\begin{eqnarray}
\frac{\dd}{\dd t}\int \frac{u_x^2}{2} \dd x &=& \int u_x u_{xt} \dd x = -\int u_{xx} u_t \dd x\nonumber\\
&=& -D \int u_{xx}^2 \dd x + \alpha\int u_{xx} \left(u \int_V F(u,x,r) \Omega(r) \dd r\right)_x \dd x \nonumber \\
&\leq & \left(-D +\frac{\alpha\ep}{2}\right) \int u_{xx}^2 \dd x +
    \frac{\alpha}{2\ep} \int\left[\left(u \int_V F(u,x,r) \Omega(r) \dd
    r\right)_x\right]^2 \dd x.\label{firstterm}
\end{eqnarray}

We continue with the second term
\begin{eqnarray*}
    && \frac{\alpha}{2\ep} \int\left[\left(u \int_V F(u,x,r) \Omega(r) \dd
    r\right)_x\right]^2 \dd x\\
&\leq&
\frac{\alpha}{2\ep}\int u_x^2 \left(\int_V F\; \Omega \dd r\right)^2 \dd x  + \frac{\alpha}{\ep} \int\left( u u_x \int_V F\; \Omega \dd r \frac{\dd}{\dd x}\int_V F \; \Omega \dd r \right) \dd x\\
&& + \frac{\alpha}{2\ep} \int u^2 \left(\frac{\dd}{\dd x} \int_V F\; \Omega \dd r\right)^2 \dd x \\
&\leq& \left(\frac{\alpha}{2\ep}+ \frac{\alpha}{2\ep}\right) \int u_x^2 \left(\int_V F\; \Omega \dd r \right)^2 \dd x  +  \left( \frac{\alpha}{2\ep} + \frac{\alpha}{2\ep} \right) \int u^2 \left(\frac{\dd}{\dd x} \int_V F\;\Omega \dd r \right)^2 \dd x\\
&\leq & C \Bigl(1+|u|_2^2\Bigr) \int u_x^2 \dd x + C \Bigl(1+||u||^2\Bigr)\int u^2 \dd x\\
&\leq & C \Bigl(1+|u|_2^2\Bigr) \Bigl( 1+||u||^2\Bigr)\\
&\leq & C \Bigl(1 + e^{2c_1 t} |u_0|_2^2\Bigr) \Bigl(1+||u||^2\Bigr).
\end{eqnarray*}
Now we choose $\ep = \frac{2D}{\alpha}$ such that the first term in (\ref{firstterm}) vanishes and we obtain
\[
    \frac{\dd}{\dd t} ||u||^2 \leq A(t) + A(t) ||u||^2,
\]
with exponentially growing coefficient function
\[
    A(t) \coloneqq C\Bigl(1+ e^{2c_1 t} |u_0|_2^2\Bigr).
\]
Hence, by Gr\"{o}nwall's Lemma,
we find that
\[
    ||u(\cdot,t)||^2 \leq \Lambda(t) ||u_0||^2 +\int_0^t \Lambda(t-s) A(s) \dd s,
    \quad \Lambda(t) = \exp\left(\int_0^t A(s) \dd s\right).
\]
Integrating $A(s)$ we find constants $c_2, c_3 >0$ such that
\[
    ||u(\cdot,t)|| \leq c_2(||u_0|| + t) e^{c_3 t}.
\]
The $H^1/\mathbb{R}$-norm, and consequently also the $L^\infty$-norm, do not
grow faster than exponential, hence the solutions are global.
\end{proof}

\section{Numerical solutions}\label{sec:numerics}

In this section we solve equation~\eqref{eq:K} numerically, for different types
of boundary conditions listed in Table~\ref{tab:examples}. We show several examples of adhesive, repulsive and neutral boundary conditions.

\subsection{Numerical methods}
Equation~\eqref{eq:K} is solved using a method of lines approach, where the
spatial derivatives are discretized to yield a large system of time-dependent
ODEs (MOL-ODEs). Towards this goal, the domain $[0, L]$ is discretized into a
cell-centered grid with uniform length $h = 1/N$, where $N$ is the number of
grid cells per unit length. We denote the cell centers as $x_i$, where $1 \leq i
\leq N_1$ (the total number of grid points). The
discretization of the advection term utilities a high-order upwinding scheme
augmented with a flux-limiting scheme to ensure positivity of solutions. For
full details on the numerical method we refer to \cite{Gerisch2001}.

A fast numerical scheme for the non-local term $\K[u]$ is a
challenge.  In the periodic case the non-local term $\K[u]$ can be efficiently
implemented using the Fast Fourier transform (FFT) \cite{Gerisch2010a}. For each
halfway point between grid points, \cite{Gerisch2010a} proposed the approximation
\[
    a_i \coloneqq \frac{1}{R} \int_{-R}^{R} \hat{g}(x_i + h/2 + r) \dd r
        = \sum_{l = 1}^{N_l} v_{i-l+1} H_l \qquad i = 1, \ldots, N_1,
\]
where $H_l$ are the weights of a piece-wise linear reconstruction of $H(x)$, and
where the coefficients $v_i$ are defined by
\[
    v_i = \frac{1}{R} \int_{-R}^{R} \Phi\lb i + \frac{r}{h}\rb \Omega(r) \dd r,
\]
where $\Phi(\cdot)$ is a piece-wise linear function.
The coefficients $v_i$ can be precomputed at the beginning of a numerical
simulation.  This means that the computation of $a_i$ can be summarized as a
matrix-vector product $\vec{a} = V \vec{H}$ for a matrix $V = (v_{il}) \in
\R^{N_1\times N_1}$. The use of the FFT to accelerate this
matrix vector product~\cite{Gerisch2010a} is well known.

However, in our case the integration limits in $\K[u]$ are spatially dependent
near the domain boundary. Thus, near the boundary the FFT can no longer be employed.
We retain the speed advantage the FFT offers, by continuing to use it far away
from the boundary (at least one sensing radius). Near the boundary we compute the
integration weights $v_i$ for
every point and use a matrix-vector product to compute the non-local term in the
boundary region. The integration weights near the boundary are given by
\[
    v_i = \frac{1}{R} \int_{f_1(x_i)}^{f_2(x_i)} \Phi\lb i + \frac{r}{h}\rb \Omega(r) \dd r.
\]

The MOL-ODEs are integrated using the ROWMAP stiff system integration~\cite{Weiner1997a}
(we use the implementation by the authors\footnote{\url{http://www.mathematik.uni-halle.de/wissenschaftliches_rechnen/forschung/software/}}).
This integrator is commonly used to integrate the possibly stiff MOL-ODEs obtained
by discretizing PDEs~\cite{Gerisch2001,Gerisch2008,Painter2009a,Hillen2013c}.

\begin{table}[!ht]
    \centering
    \begin{tabularx}{0.75\textwidth}{@{} >{}Y p{0.3\textwidth} @{}}
        \toprule
        Model Parameter & Value \\
        \midrule
        Domain Size $L$ & 5.0 \\
        Domain subdivisions per unit length & 128 \\
        Diffusion coefficient $D$ & 1.0  \\
        Adhesion strength coefficient $\alpha$ & varies \\
        Sensing radius $R$ & 1.0 \\
        Initial conditions (IC) & $1 + \xi$, $\xi \sim \mathcal{N}(0, 1)$ \\
        Method error tolerance $v_{tol}$ & $10^{-5}$ \\
        Final simulation time $t_f$ & $25$ \\
        \bottomrule
    \end{tabularx}
    \caption{Common parameters for the numerical solutions.}\label{Tab:NumericalParameters}
\end{table}

\subsection{Solutions on a periodic domain}
As a control case, we show typical solutions of equation~\eqref{Armstrong1} on a
periodic domain first. In this case we use the sensing domain $E_1(x)$ (see
\cref{Example:Periodic}). An extensive bifurcation analysis of the periodic case is given in \cite{AMS-memoires} and we know
that the first three bifurcation points from the homogeneous solution are located at
\[
    \alpha_1 = \frac{16\pi^2}{25(5 - \sqrt{5})}, \qquad
    \alpha_2 = \frac{64\pi^2}{25(5 + \sqrt{5})}, \quad
    \alpha_3 = \frac{144\pi^2}{25(5 + \sqrt{5})}.
\]
This roughly means that $\alpha_1 \sim 2.28$, $\alpha_2 \sim 3.49$,
$\alpha_3 \sim 7.85$. For all subsequent numerical simulations we pick a value
of $\alpha$ from each of the intervals $(0, \alpha_1)$, $(\alpha_1, \alpha_2)$,
and $(\alpha_2, \alpha_3)$. The numerical solutions of equation~\eqref{Armstrong1}
with periodic boundary conditions are shown in \cref{Fig:PeriodicNumericalResults}.
We identify three important features in these solutions. Firstly, for values of $\alpha$
below the first bifurcation point, the solution is constant. Secondly, as predicted by
the bifurcation analysis in~\cite{AMS-memoires}, a single peak forms between the first
and second bifurcation point. Finally, we note that due to the translational
symmetry permitted by the periodic boundary conditions, the solution peak may
form at any location within the domain. Higher bifurcation points lead to a
larger number of aggregations in the domain. The local maxima have a uniform
distance and they can arise anywhere in the domain due to rotational symmetry
\cite{AMS-memoires}.

\begin{figure}[!ht]
    \includegraphics[width=0.95\textwidth]{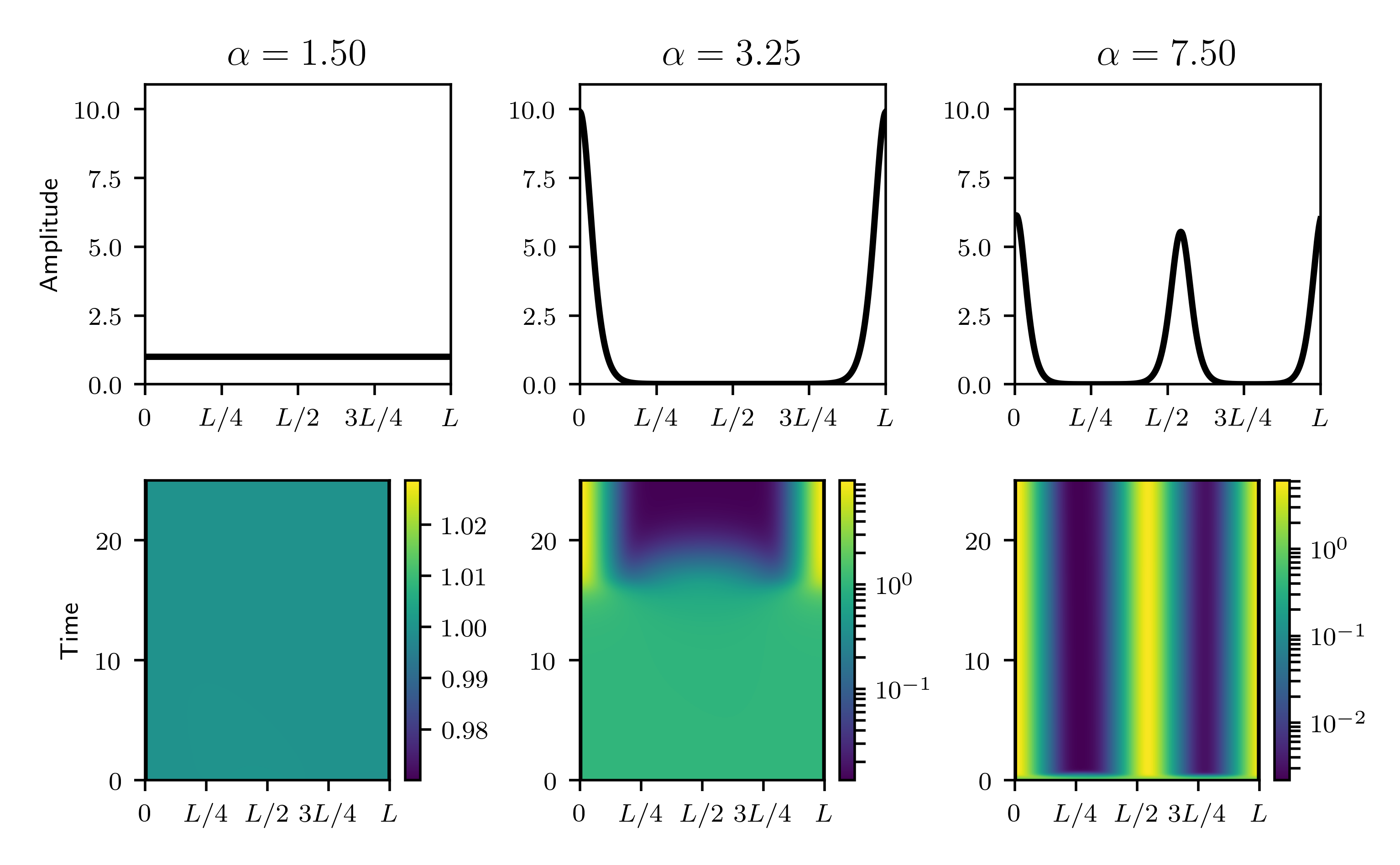}
    \caption{Numerical solutions of equation~\eqref{Armstrong1} subject to periodic boundary
    conditions (see \cref{Example:Periodic}). In the top row we show the final solution
    profiles, while below are the kymographs.
    (Left) $\alpha = 1.5$, (Middle) $\alpha = 3.25$, (Right) $\alpha = 7.5$.
    }\label{Fig:PeriodicNumericalResults}
\end{figure}

\subsection{Solutions with No-Flux boundary conditions}

We compute numerical solutions for equation~\eqref{Armstrong1} with the no-flux  sensing
domain $E_3(x)$ (see \cref{Example:NoFlux}). The numerical solutions are shown
in \cref{Fig:NoFluxNumericalResults}. Comparing these no-flux solutions to the periodic
solution in \cref{Fig:PeriodicNumericalResults} we identified three differences.
First, for $0 < \alpha < \alpha_1$ the solution is no longer constant. In fact the
constant solution is now only a solution for $\alpha = 0$. In particular,
we note that the solution decreases near the boundary, indicating that the boundary
is repulsive. The repulsive nature of the boundaries appears to speed up
peak formation in the no-flux case, compared to the periodic case.
Second for $\alpha > \alpha_1$ the final no-flux solution profiles
correspond to  those in the periodic case. Since the no-flux boundary conditions
break the translational symmetry observed in the periodic case, the locations of the
peaks are fixed in the no-flux case. Finally, we note that while the bifurcation
analysis carried out in~\cite{AMS-memoires} cannot be straightforwardly extended
to the no-flux situation, the numerical results suggest that the bifurcation
points are similar.

\begin{figure}[!ht]
    \includegraphics[width=0.95\textwidth]{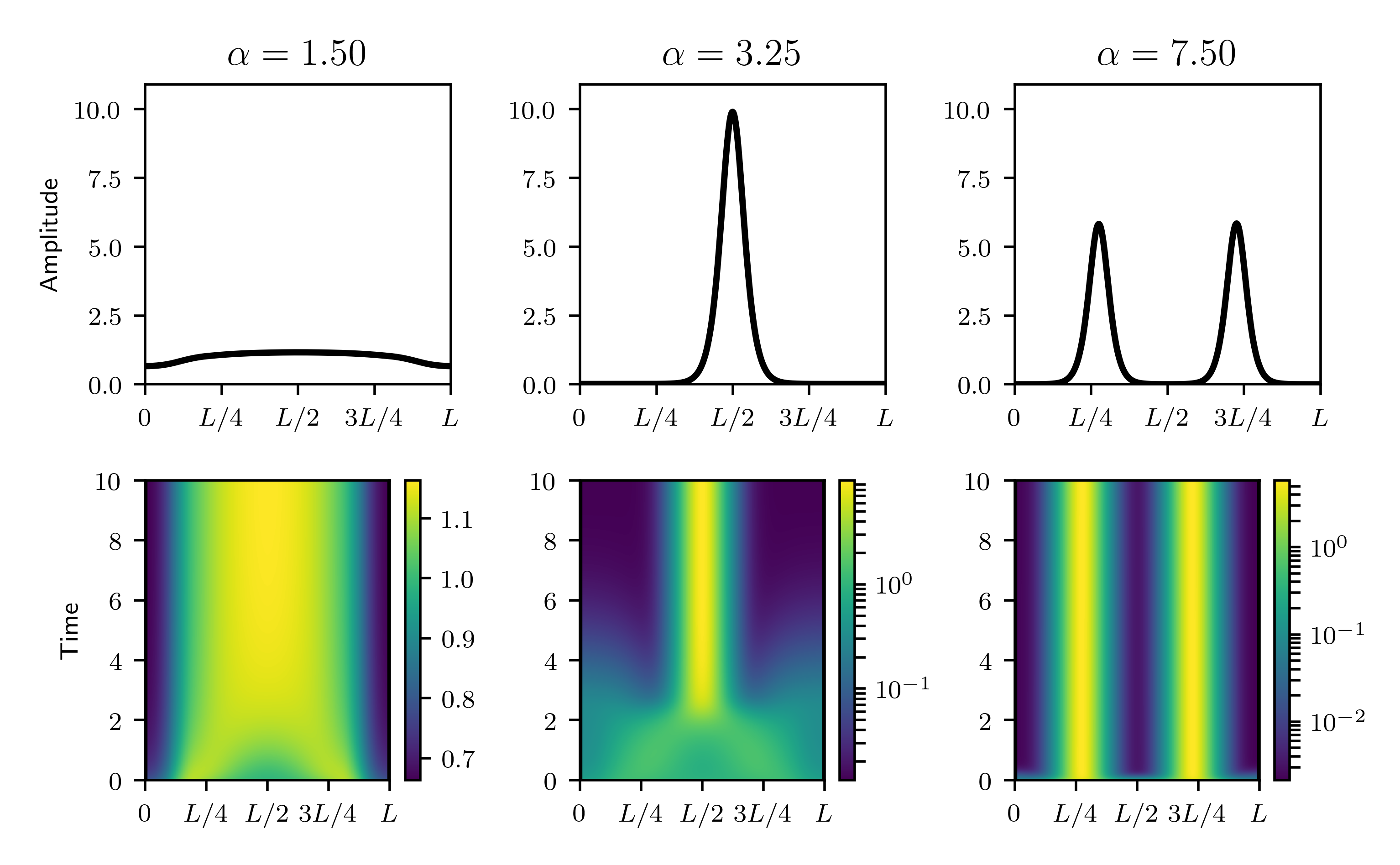}
    \caption{Numerical solutions of equation~\eqref{Armstrong1} subject to no-flux boundary
    conditions with $E_3(x)$ (see \cref{Example:NoFlux}).
    In the top row we show the final solution
    profiles, while below are the kymographs.
    (Left) $\alpha = 1.5$, (Middle) $\alpha = 3.25$, (Right) $\alpha = 7.5$.
    }\label{Fig:NoFluxNumericalResults}
\end{figure}

%
%
%
%
\subsection{Solutions of the adhesive and repulsive boundaries}

In this section we demonstrate numerical solutions of with the so called weakly
adhesive boundary conditions i.e.\ sensing domain $E_4(x)$ (see
\cref{Example:WeaklyAdhesive}). In particular, we consider the situation in which
the constructed $\K[u]$ does not satisfy condition~\eqref{nofluxK} i.e.\ the dependent
case. Since in this case the adhesive and diffusive fluxes depend on each other,
the existence of solutions in this case are not included in the theoretical results
in this paper. However, since we can compute solutions in this case numerically,
we explore their possible solutions numerically. As before we distinguish between
two types of boundaries (1) attractive boundaries $\beta > 0$ and
(2) repulsive boundaries $\beta < 0$. The numerical solutions are shown in
\cref{Fig:WeaklyAdhesiveAttractivesSoln} and \cref{Fig:WeaklyAdhesiveRepusliveSoln}
respectively.

When the adhesive strength is weak, $\alpha < \alpha_1$, we note that the solution either accumulates (adhesive boundary)
or is repelled from the boundary (repulsive boundary), while far away from the boundaries the
solutions are constant. These solutions are reminiscent of the menisci which form at a
liquid solid interface (e.g.\ water-glass or mercury-glass). It is well known that the
meniscus is concave whenever the liquid-solid adhesion is stronger than liquid-liquid
cohesion, while it is convex whenever liquid-liquid cohesion is weaker than liquid-solid
adhesion.

For stronger adhesive strength,  $\alpha > \alpha_1$, we once again observe the formation of peaks with fixed locations.
In the case with adhesive boundary conditions we always find two half peaks on the boundary,
while in the repulsive boundary case both peaks form in the domain's interior. Once again
the periodic bifurcation analysis appears to be a good predictor of the bifurcation
points with different boundary conditions.

\begin{figure}[!ht]
    \includegraphics[width=0.95\textwidth]{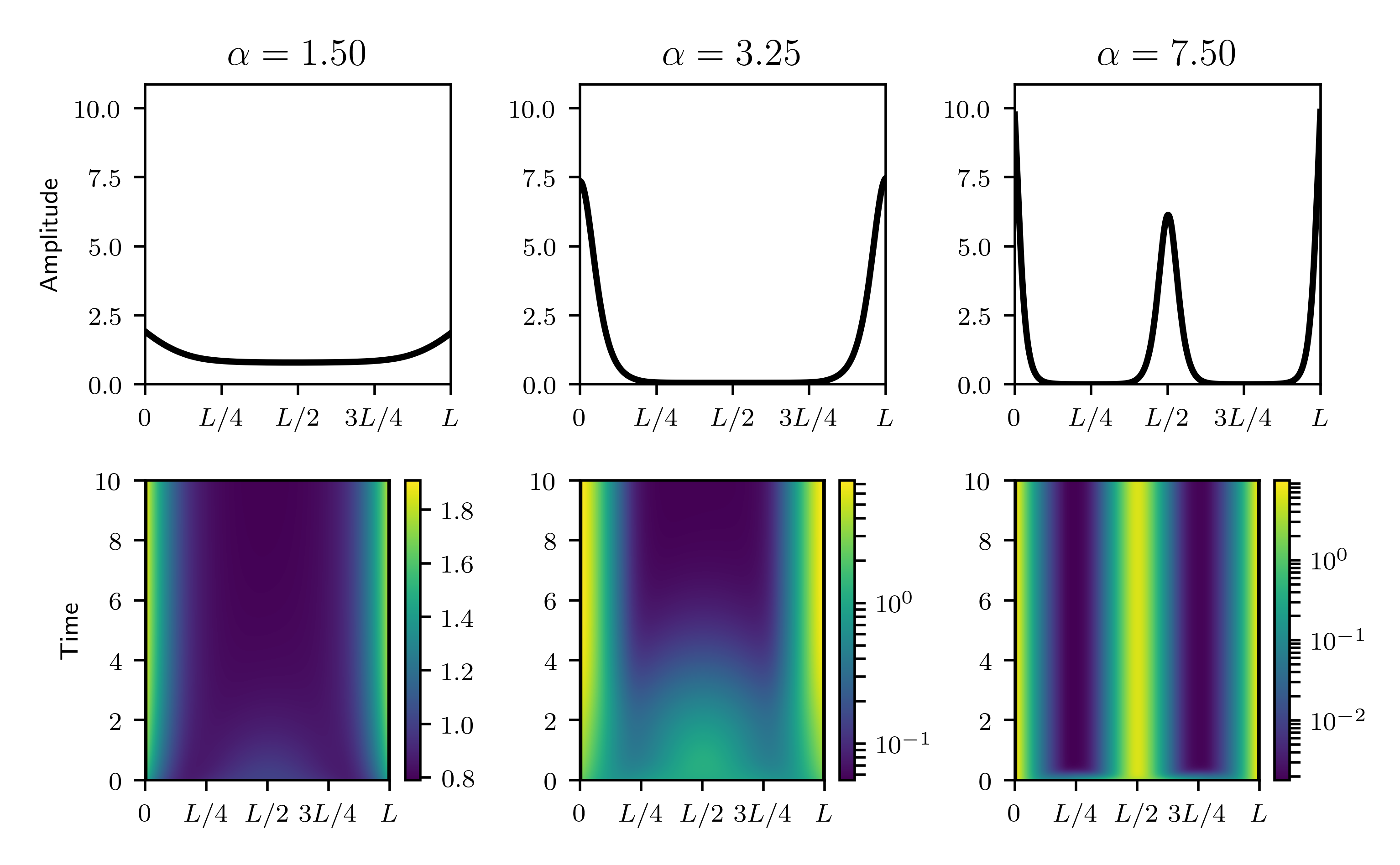}
    \caption{Numerical solutions of equation~\eqref{Armstrong1} subject to
    adhesive boundary conditions (see \cref{Example:WeaklyAdhesive}) with $\beta = 2$.
    In the top row we show the final solution
    profiles, while below are the kymographs.
    (Left) $\alpha = 1.5$, (Middle) $\alpha = 3.25$, (Right) $\alpha = 7.5$.
    }\label{Fig:WeaklyAdhesiveAttractivesSoln}
\end{figure}

\begin{figure}[!ht]
    \includegraphics[width=0.95\textwidth]{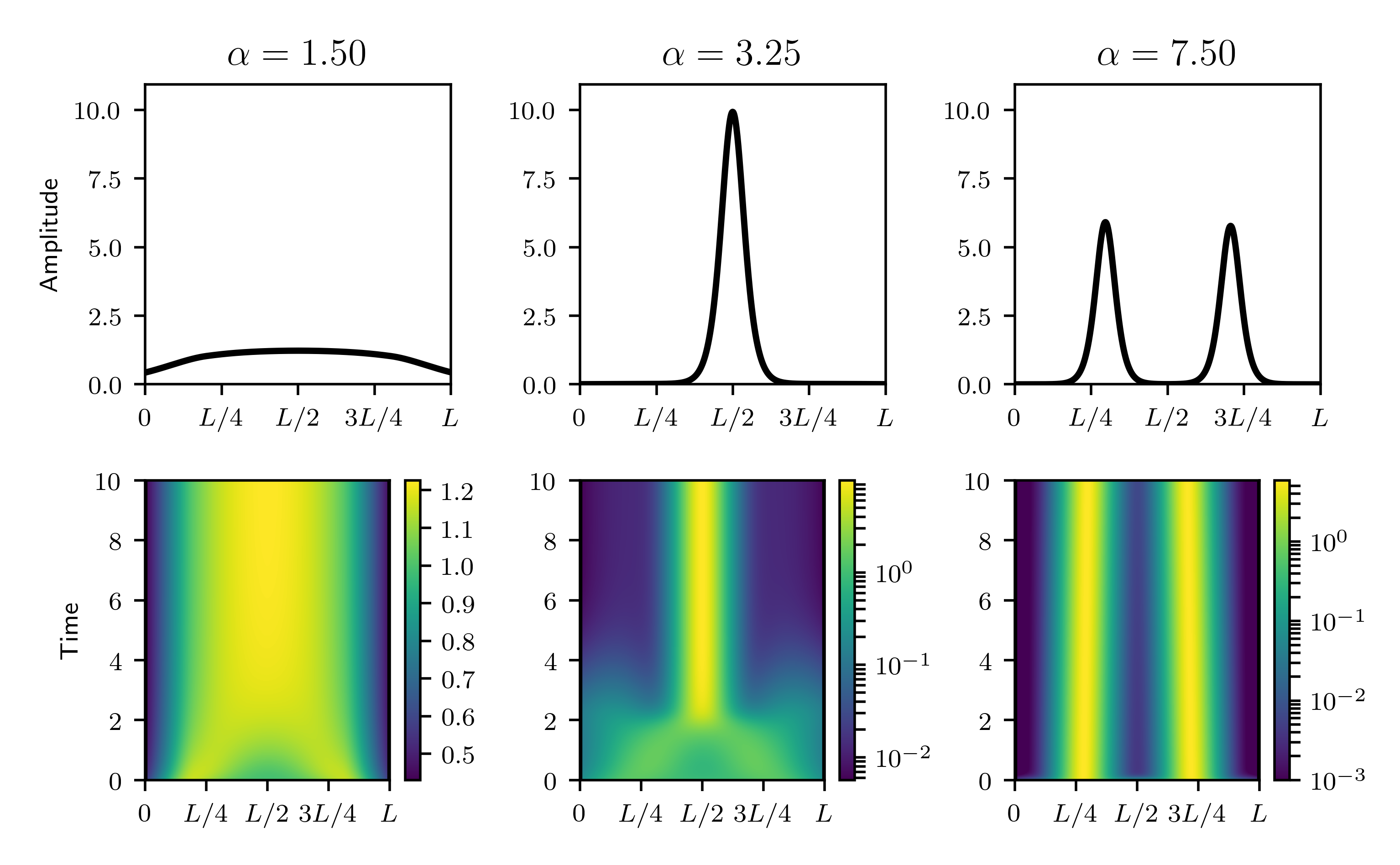}
    \caption{
    Numerical solutions of equation~\eqref{Armstrong1} subject to
    repulsive boundary conditions (see \cref{Example:WeaklyAdhesive}) with $\beta = -1$.
    In the top row we show the final solution
    profiles, while below are the kymographs.
    (Left) $\alpha = 1.5$, (Middle) $\alpha = 3.25$, (Right) $\alpha = 7.5$.
    }\label{Fig:WeaklyAdhesiveRepusliveSoln}
\end{figure}

\section{Conclusions}\label{sec:conclusion}

In the past due to
the challenges in construction of boundary conditions, the non-local adhesion
model was only considered on unbounded domains or with periodic boundary
conditions. However, correct adhesive-boundary interactions are important
in biological systems such as during zebrafish development.
Here we considered the formulation of no-flux boundary conditions for
the non-local adhesion model~\eqref{Eqn:ArmstrongModelIntro}, and established
the global existence and uniqueness of solutions of~\eqref{Armstrong1}.
Thus our work here significantly extends our methods of modelling cell adhesion.
We considered two possible methods of extending the non-local adhesion
operator (1) by treating the adhesion and diffusion flux as independent, and
(2) having the two fluxes depend on each other.

In the independent flux case, we
impose standard Neumann boundary conditions for the cell population $u(x,t)$,
while the behaviour of the non-local operator $\K$ near the boundary is
built into the operator itself. For these no-flux boundary conditions, we
establish the global existence of solutions, using standard methods from
semi-group theory. While the argument itself is standard, it relies on
the novel computation of the weak derivative of the non-local term and its
estimates.

The numerical solutions demonstrate that due to the no-flux boundary conditions
the translational symmetry observed in the periodic case is broken, and that
peaks form at precisely defined locations. This is significant in many
biological systems in which combinations of repellent boundaries together with
cell-cell adhesion are used to precisely position pre-cursor cells of
organs~\cite{Paksa2016}. Repulsive boundaries also accelerated the formation of
adhesive cell clusters away from the boundary.

Our existence theory is currently limited to
the situation in which the diffusive and adhesive flux are independently
zero on the domain's boundary. In particular, the adhesive / repulsive
boundary conditions from \cref{Example:WeaklyAdhesive} are not covered
by our theory except for one particular choice for $\beta$. It is therefore
highly desirable to extend the existence theory to include the cases of \cref{Example:WeaklyAdhesive}.
In this case, we must solve a non-local equation~\eqref{Eqn:ArmstrongModelIntro},
subject to non-local Robin boundary conditions (\ref{Robin}).
This is a challenging problem. A starting point may be the recent work by
\cite{Arendt2018}, who studied the semi-group originating from an
elliptic operator on a bounded domain with a linear non-local Robin type boundary
condition. As our Robin condition (\ref{Robin}) is non-linear,
the methods of \cite{Arendt2018} will not directly apply and non-linear methods need to be developed.

\section*{Acknowledgments}
AB gratefully acknowledges support from a NSERC post-doctoral fellowship.
TH gratefully acknowledges support from an NSERC discovery grant.

\bibliographystyle{siamplain}
\bibliography{references}
\end{document}